\documentclass[11pt,a4paper,reqno]{amsart}
\usepackage{amsmath,amsfonts,amssymb,amsthm,mathtools,mathrsfs}
\usepackage[dvipsnames]{xcolor}
\usepackage{booktabs,array,microtype,cite}
\usepackage[textwidth=14.5cm,textheight=21cm,centering]{geometry}
\usepackage[pagebackref=true,colorlinks=false]{hyperref}
\hypersetup{pdfborder={0 0 0.65},linkbordercolor={1 0 0},
	citebordercolor={0 0.65 0},urlbordercolor={0 0 1},
	pdftitle={A New Bound for the Gauss Circle Problem},
	pdfauthor={Xi Cen, Xitao Gao, Cheng Zhang, Junyong Zhang},
	pdfsubject={Gauss circle problem}}
\renewcommand*{\backref}[1]{}

\renewcommand*{\backrefalt}[4]{\ifcase #1\relax\else\space{\footnotesize #2}\fi}

\newcommand{\arxiv}[1]{\href{https://arxiv.org/abs/#1}{arXiv\,\nolinkurl{#1}}}
\DeclareRobustCommand{\Extref}[2]{\mbox{\hyperlink{cite.#1}{#2}}}
\DeclareRobustCommand{\Thmref}[1]{\mbox{\hyperref[#1]{Theorem~\ref*{#1}}}}
\DeclareRobustCommand{\Lemref}[1]{\mbox{\hyperref[#1]{Lemma~\ref*{#1}}}}
\DeclareRobustCommand{\Propref}[1]{\mbox{\hyperref[#1]{Proposition~\ref*{#1}}}}
\DeclareRobustCommand{\Corref}[1]{\mbox{\hyperref[#1]{Corollary~\ref*{#1}}}}
\DeclareRobustCommand{\Remref}[1]{\mbox{\hyperref[#1]{Remark~\ref*{#1}}}}
\DeclareRobustCommand{\Secref}[1]{\mbox{\hyperref[#1]{Section~\ref*{#1}}}}
\DeclareRobustCommand{\Subsecref}[1]{\mbox{\hyperref[#1]{Subsection~\ref*{#1}}}}
\DeclareRobustCommand{\Appref}[1]{\mbox{\hyperref[#1]{Appendix~\ref*{#1}}}}
\allowdisplaybreaks[2]
\newcommand{\R}{\mathbb R}
\newcommand{\Z}{\mathbb Z}
\newcommand{\T}{\mathbb T}
\newcommand{\N}{\mathbb N}
\newcommand{\Var}{\operatorname{Var}}
\numberwithin{equation}{section}
\newtheorem{theorem}{Theorem}[section]
\newtheorem{lemma}{Lemma}[section]
\newtheorem{proposition}{Proposition}[section]

\newtheorem{conjecture}{Conjecture}[section]
\theoremstyle{definition}
\theoremstyle{remark}
\begin{document}
	\title[Gauss circle problem]
	{A New Bound for the Gauss Circle Problem}
	\date{}
	
	\author[X. Cen]{Xi Cen}
	\address{Xi Cen, School of Science, China University of Mining and Technology-Beijing, Beijing 100083, People's Republic of China}
	\email{xicenmath@gmail.com}
	
	\author[X. Gao]{Xitao Gao}
	\address{Xitao Gao, Department of Mathematics, Beijing Institute of Technology, Beijing 100081, People's Republic of China}
	\email{xitao\_gao@bit.edu.cn}
	
	\author[C. Zhang]{Cheng Zhang}
	\address{Cheng Zhang, Yau Mathematical Sciences Center, Tsinghua University, Beijing, 100084, P.R. China}
	\email{czhang98@tsinghua.edu.cn}
	
	\author[J. Zhang]{Junyong Zhang}
	\address{Junyong Zhang, Department of Mathematics, Beijing Institute of Technology, Beijing 100081, People's Republic of China}
	\email{zhang\_junyong@bit.edu.cn}
	
	\begin{abstract}
	Let $E(X)=\#\{(a,b)\in\Z^2:a^2+b^2\le X\}-\pi X$.
	We prove that $E(X)=O_\varepsilon(X^{209/670+\varepsilon})$
	for every $\varepsilon>0$, uniformly for real $X\ge4$.
	This breaks the $5/16$ threshold suggested by Huxley as a possible limit of the Bombieri--Iwaniec method at the time.
	The proof introduces two refinements: a cone distribution estimate
	that preserves separate fourth- and sixth-power contributions
	in the double large sieve, and an explicit inversion of
	Taylor-coefficient constraints that cancels the linear error in a pair-counting argument.
These refinements, together with the Robert--Sargos and Watt spacing estimates and the Vinogradov mean value theorem, yield the improvement.
	\end{abstract}
	\maketitle
	
	\section{Introduction}\label{sec:intro}
	\subsection{Main result}
	For real $X\ge0$, define the circle remainder by
	\[
	E(X)=\#\{(a,b)\in\Z^2:a^2+b^2\le X\}-\pi X,
	\]
	where $\Z$ denotes the integers and $\#$ the cardinality of a finite set.
	The Gauss circle problem asks for the sharp pointwise bound on $E(X)$.
	Gauss's geometric argument gives $E(X)=O(X^{1/2})$
	\cite[p.~280]{Gauss}, whereas the conjectured bound is
	$O_\varepsilon(X^{1/4+\varepsilon})$.
	Successive improvements have been closely tied to the development of
	exponential-sum methods \cite{Kratzel88,GK,IKKN,BKZ18}.
	For nonnegative $A,B$, we write $A\lesssim B$ if $A\le CB$ for a constant
	$C>0$. A subscript indicates the permitted dependence of $C$.
	Our main result is the following.
	
	\begin{theorem}\label{thm:main}
		For every $\varepsilon>0$ and every real $X\ge4$,
		\[
		|E(X)|\lesssim_\varepsilon X^{209/670+\varepsilon}.
		\]
	\end{theorem}
	
\Thmref{thm:main} improves the exponent
$0.3144831759741\ldots$ of Li and Yang
\cite{LY23}
and breaks the $5/16$ threshold:
\[
\frac{209}{670}=\frac5{16}-\frac3{5360}.
\]
Huxley \cite[p.~280]{Huxley93} suggested this threshold
as a possible limit of the Bombieri--Iwaniec method at the time.
See also Popov \cite[p.~914]{Popov19}, \cite[p.~55]{Popov24}
and Li and Yang \cite[end of \Extref{LY23}{Section~1}]{LY23}.
	
	The improvement uses two complementary refinements.
	From Guth--Maldague's amplitude-dependent estimate \cite[\Extref{GM}{Theorem~2}]{GM},
	we obtain a distribution estimate for cone exponential sums whose
	fourth- and sixth-power terms remain separate through the double large
	sieve, rather than
	first passing to a single moment estimate (\Thmref{thm:first}).
	In the complementary range, explicit inversion of the
	Taylor-coefficient constraints cancels the linear error and yields
	a pair count of diagonal size, up to arbitrarily small power losses
	(\Propref{c:count-prop}). Together these refinements give
	$209/670$ without strengthening Watt's second-spacing bound
	\cite[\Extref{Watt04}{Lemma~8.2}]{Watt04}.
	
	\subsection{Historical upper bounds}
	The table records pointwise bounds
	$E(X)=O_\varepsilon(X^{\theta+\varepsilon})$, including independent
	proofs as well as exponent improvements. All exponents use the squared
	radius $X$, so the corresponding radius exponent is $2\theta$.
	Logarithmic losses are suppressed. Dates are publication dates, except
	for the marked preprint. For early methods and historical accounts, see
	\cite{Walfisz26,vdC28,Wilton28,Kratzel88,LuneWattel90,BKZ18}.
	\begin{center}\small
		\setlength{\tabcolsep}{5pt}
		\begin{tabular}{@{}p{8.5cm}lc@{}}\toprule
			Work & Year & Exponent in $X$\\\midrule
			Sierpi\'nski \cite{Sier06} & 1906 & $1/3$\\
			Landau \cite{Landau12,Landau13} & 1912--1913 & $1/3$\\
			Van der Corput \cite{vdC23} & 1923 & $33/100$\\
			Littlewood--Walfisz, Landau \cite{LittlewoodWalfisz24} & 1924 & $37/112$\\
			Hardy \cite{Hardy25} & 1925 & $37/112$\\
			Walfisz \cite{WalfiszTeiler27} & 1927 & $163/494$\\
			Nieland (van der Corput's method) \cite{Nieland28} & 1928 & $27/82$\\
			Titchmarsh \cite{Titchmarsh34} & 1934 & $15/46$\\
			Hua \cite{Hua42} & 1942 & $13/40$\\
			Yin, with Chen's correction \cite{Yin62} & 1962 & $12/37$\\
			Chen \cite{Chen63} & 1963 & $12/37$\\
			Nowak \cite{Nowak84} & 1984 & $35/108$\\
			Kolesnik \cite{Ko85} & 1985 & $139/429$\\
			Iwaniec--Mozzochi \cite{IM88} & 1988 & $7/22$\\
			Huxley \cite{Huxley93} & 1993 & $23/73$\\
			Huxley \cite{Huxley03} & 2003 & $131/416$\\
			Li--Yang \cite{LY23} (preprint) & 2023 & $0.3144831759\ldots$\\
			\Thmref{thm:main} & & $209/670$\\\bottomrule
		\end{tabular}
	\end{center}
	
	Several advances arose first in the divisor problem. Kolesnik obtained
	the exponents $12/37$, $346/1067$ and $35/108$ in 1969, 1973 and 1982,
	respectively \cite{Ko69,Ko73,Ko82}. Nowak \cite{Nowak84} transferred
	the last of these to the circle problem through divisors in arithmetic
	progressions. Kolesnik's later exponent-pair refinement gave $139/429$
	\cite{Ko85}.
	
	The Bombieri--Iwaniec method \cite{BI86}, adapted to the circle and
	divisor problems by Iwaniec--Mozzochi \cite{IM88}, introduced a new
	approach based on rational approximation and the double large sieve.
	It reduces the exponential sums to first- and second-spacing problems.
	Huxley's development of this method
	\cite{Huxley90,Huxley93,Huxley96,Huxley03} culminated in
	\cite{Huxley03,IKKN}
	\[
	E(X)=O\!\left(X^{131/416}(\log X)^{18637/8320}\right)
	\qquad(X\ge4).
	\]
	
	Bourgain--Watt \cite{BW18} connected the first-spacing problem with
	decoupling for perturbed cones. Li--Yang \cite{LY23} then combined
	Guth--Maldague's small-cap decoupling theorem \cite{GM} with
	second-spacing estimates to improve Huxley's exponent.
	
	\subsection{Lower bounds and conjectures}
	Write $\log$ for the natural logarithm and, for sufficiently large $X$,
	$\log_2X=\log\log X$ and $\log_3X=\log\log\log X$.
	For a positive function $G$, the notation $E(X)=\Omega(G(X))$ means
	$\limsup_{X\to\infty}|E(X)|/G(X)>0$.
	Replacing $|E|$ by $E$ or $-E$ defines $\Omega_+$ or $\Omega_-$,
	respectively. We write $\Omega_\pm$ when both signs hold.
	Hardy \cite{Hardy15,Hardy16} and Landau \cite{Landau15} established
	the $X^{1/4}$ scale of oscillation. In particular, Hardy proved
	\[
	E(X)=\Omega_\pm(X^{1/4}),\qquad
	E(X)=\Omega_-\bigl((X\log X)^{1/4}\bigr).
	\]
	Thus $E(X)=O(X^{1/4})$ is false, but the following power bound is expected.
	\begin{conjecture}[Hardy's conjecture]\label{conj:Hardy}
		For every $\varepsilon>0$ and every real $X\ge4$,
		\[
		|E(X)|\lesssim_\varepsilon X^{1/4+\varepsilon}.
		\]
	\end{conjecture}
	
	Further lower-bound methods were developed by Ingham \cite{Ingham40}
	and Gangadharan \cite{Gangadharan61}.
	Corr\'adi--K\'atai \cite{CK67} and Hafner \cite{Hafner81} proved,
	respectively, that for some constants $c,C>0$,
	\[
	\begin{aligned}
		E(X)&=\Omega_+\!\left(X^{1/4}
		\exp\!\left(c(\log_2X)^{1/4}(\log_3X)^{-3/4}\right)\right),\\
		E(X)&=\Omega_-\!\left((X\log X)^{1/4}(\log_2X)^{(\log2)/4}
		\exp\!\left(-C\sqrt{\log_3X}\right)\right).
	\end{aligned}
	\]
	For $X>\exp(1)$, set
	\[
	\mathcal L(X)=(X\log X)^{1/4}
	(\log_2X)^{\frac34(2^{1/3}-1)}.
	\]
	Soundararajan \cite[\Extref{S03}{Theorem~1}]{S03} obtained
	\[
	E(X)=\Omega\!\left(\mathcal L(X)(\log_3X)^{-5/8}\right).
	\]
	Sourmelidis \cite{Sourmelidis26} and Mahatab \cite{Mahatab26}
	independently improved the exponent $-5/8$ to $-3/8$ by the resonance
	method. Lamzouri \cite[\Extref{Lam26}{Theorem~1.1}]{Lam26} removed the remaining
	$\log_3X$ loss and determined a sign:
	\[
	E(X)=\Omega_-\bigl(\mathcal L(X)\bigr).
	\]
	This reaches the scale predicted by \cite[\Extref{Lam25}{Conjecture~1.7}]{Lam25}.
	\begin{conjecture}[Lamzouri's maximal-order conjecture]\label{conj:maximal}
		There are constants $c_1,c_2>0$ such that, for either sign and all
		sufficiently large $X$,
		\[
		c_1\mathcal L(X)\le\sup_{X\le x\le2X}(\pm E(x))
		\le c_2\mathcal L(X).
		\]
	\end{conjecture}
	This implies Conjecture~\ref{conj:Hardy}.
	The negative $\Omega$-result concerns an unbounded sequence of arguments,
	not both signs on every dyadic interval as required here.
	
	\subsection{Mean values and distribution}
	Cram\'er \cite{Cramer22} established a mean-square asymptotic with a
	positive leading constant. Its remainder was improved by
	Landau \cite{Landau23}, Walfisz \cite{WalfiszTeiler27} and
	K\'atai \cite{Katai65}. Preissmann \cite{Preissmann88} gave a shorter
	proof of K\'atai's estimate. Nowak \cite{Nowak04} proved that, for some
	$C>0$,
	\[
	\int_0^X E(t)^2\,dt
	=CX^{3/2}+O\!\left(X(\log X)^{3/2}\log\log X\right).
	\]
	This supports the $X^{1/4}$ scale of typical fluctuations, but does not
	control the maximal order.
	Heath-Brown \cite{HeathBrown92} studied limiting distributions and
	moments, and Tsang \cite{Tsang92} obtained higher-moment asymptotics.
	Lester--Wigman \cite{LW23} examine correlations of lattice points near
	circles, while Lamzouri \cite{Lam25} studies large deviations of the
	circle remainder.
	
	\subsection{Notation}
	We write $\N=\{1,2,\ldots\}$. For $z\in\mathbb C$ and $t\in\R$, set
	\[
	e(z)=\exp(2\pi\mathrm i z),\qquad
	\|t\|=\min_{j\in\Z}|t-j|.
	\]
	For a set $\mathcal U$, $\mathbf1_{\mathcal U}$ is its indicator.
	The floor and ceiling are characterized by
	$\lfloor t\rfloor\le t<\lfloor t\rfloor+1$ and
	$\lceil t\rceil-1<t\le\lceil t\rceil$, and
	$\{t\}=t-\lfloor t\rfloor$ is the fractional part.
	Sums over real intervals run over their integer points.
	Unless stated otherwise, $\sum_{n\le Y}$ means $\sum_{1\le n\le Y}$,
	and $\sum_{d\mid n}$ runs over the positive divisors of $n\in\N$.
	We write $n\approx M$ for $M\le n<2M$.
	A dyadic scale is $2^j$, $j\in\Z$, unless another base is specified.
	Each dyadic interval may be intersected with an arbitrary interval,
	with either endpoint included or excluded.
	
	For nonnegative $A,B$, write $A\eqsim B$ if $A\lesssim B\lesssim A$,
	and $A\gtrsim B$ if $B\lesssim A$.
	The notation $A\sim B$ is reserved for asymptotic equivalence,
	meaning $A/B\to1$ in the specified limit.
	The notation $O(B)$ denotes a quantity of modulus $\lesssim B$.
	Subscripts on $O$ and $\lesssim$ have the same meaning.
	Every estimate involving $\varepsilon$ holds for each $\varepsilon>0$.
	Positive constants $c,C$ may change from line to line and may depend
	on fixed hypotheses, but not on varying parameters.
	Parameters are local to a statement or proof unless explicitly retained
	for later sections.
	
	For a complex-valued function $b$ on an interval $I$, set
	\[
	\|b\|_{\infty,I}=\sup_{x\in I}|b(x)|,\qquad
	\Var_I(b)=\sup_{\substack{j\ge1\\x_0<\cdots<x_j,\ x_i\in I}}
	\sum_{i=1}^j|b(x_i)-b(x_{i-1})|.
	\]
	These use pointwise values, including values at discontinuities.
	For integers $u\le v$ and a finite sequence $b=(b_n)_{n=u}^v$, set
	\[
	\|b\|_\infty=\max_{u\le n\le v}|b_n|,\qquad
	\Var(b)=\sum_{n=u}^{v-1}|b_{n+1}-b_n|.
	\]
	For a measurable set $\mathcal U\subseteq\R^d$, $d\in\N$,
	let $|\mathcal U|$ denote its Lebesgue measure.
	For a measurable box $\mathcal B\subseteq\R^d$ of positive finite volume
	and $1\le q<\infty$, write
	\[
	\|f\|_{L^q_\#(\mathcal B)}
	=\left(\frac1{|\mathcal B|}\int_{\mathcal B}|f(x)|^q\,dx\right)^{1/q}.
	\]
	On $\R^d$, $\|f\|_q$ denotes the usual Lebesgue norm.
	For $x,\xi\in\R^d$, write $x\cdot\xi=\sum_{i=1}^dx_i\xi_i$,
	$|x|_\infty=\max_{1\le i\le d}|x_i|$, and let $|x|$ be the Euclidean norm.
	Set $\T=\R/\Z$, with Lebesgue measure of total mass one.
	For integrable $f,g$ on $\R^d$ and an integrable one-periodic function $v$,
	our Fourier and convolution conventions are
	\[
	\begin{gathered}
		\widehat f(\xi)=\int_{\R^d}f(x)e(-x\cdot\xi)\,dx,\qquad
		(f*g)(x)=\int_{\R^d}f(x-y)g(y)\,dy,\\
		\widehat v(j)=\int_0^1v(t)e(-jt)\,dt\qquad(j\in\Z).
	\end{gathered}
	\]
	For $v:\mathcal U\to\mathbb C$ and $\lambda>0$,
	$\{|v|>\lambda\}=\{x\in\mathcal U:|v(x)|>\lambda\}$.
	For a matrix $A$, $A^{\mathsf T}$ is its transpose and, when $A$ is
	invertible, $A^{-\mathsf T}=(A^{-1})^{\mathsf T}$.
	We write $\mathscr S(\R^d)$ for the Schwartz class and
	$\operatorname{supp}v$ for the support of $v$.
	Logarithmic factors are absorbed into arbitrarily small power losses
	using
	\[
	(\log(2T))^A\le C_{A,\varepsilon}T^\varepsilon
	\qquad(A\ge0,\ \varepsilon>0,\ T\ge2).
	\]
	
	\subsection{The reciprocal-sum estimate}
	Let $T\ge2$, $H\ge1$, $1\le M\le T^{1/2}$, and let
	$\alpha,\kappa\in\R$ satisfy $|\alpha|\le1/4$ and $|\kappa|\le1$.
	Let $g,G:[1,2]\to\mathbb C$ satisfy
	\[
	\|g\|_{\infty,[1,2]}+\Var_{[1,2]}(g)\le1,\qquad
	\|G\|_{\infty,[1,2]}+\Var_{[1,2]}(G)\le1.
	\]
	Define $F(z)=1/[4(z+\alpha)]+\kappa$ for $z>-\alpha$, and put
	\begin{equation}\label{eq:S}
		S=\sum_{h\approx H}g(h/H)\sum_{n\approx M}G(n/M)
		e\!\left(\frac{hT}{M}F(n/M)\right).
	\end{equation}
	Each dyadic block may be intersected with an arbitrary interval,
	with either endpoint included or excluded. Set
	\[
	\beta=\frac{209}{670},\qquad s=\frac72\beta-1=\frac{123}{1340}.
	\]
	\begin{theorem}\label{thm:recip}
		For every $\varepsilon>0$, uniformly for this family,
		\begin{equation}\label{eq:recipmain}
			|S|\lesssim_\varepsilon HT^{\beta+\varepsilon},\qquad
			1\le H\le MT^{-\beta},\quad M\le T^{1/2}.
		\end{equation}
	\end{theorem}
	Define $\log_T x=(\log x)/(\log T)$ for $x>0$, and set
	\begin{equation}\label{eq:params}
		\rho=\frac{HT}{M},\quad U=\frac{HT}{M^2},\quad
		p=\log_T\rho,\quad m=\log_T M,\quad \tau=1-p.
	\end{equation}
	Thus $H=T^{m+p-1}$, $U=T^{p-m}$, $\rho^2=THU$, and
	$U/H=T/M^2\ge1$.
	The definitions in this subsection are retained throughout
	Sections~\ref{sec:main-proof}--\ref{sec:secondspacing}.
	Constants in estimates for \eqref{eq:S} are independent of
	$T,H,M,\alpha,\kappa$, the weights and the interval endpoints.
	Auxiliary estimates may depend on their stated comparison constants
	and, when indicated, on $\varepsilon$ or the dimension.
	In these sections, $o(1)$ refers to $T\to\infty$, uniformly in the
	stated parameter region. Other limits are specified explicitly.
	For fixed $T_0\ge2$, the range $2\le T\le T_0$ follows from
	$|S|/H\lesssim M\le T_0^{1/2}$.
	
	\subsection{Proof outline}
	It suffices to prove \Thmref{thm:recip}.
	\Secref{sec:main-proof} makes this reduction using the classical
	circle-problem identity \cite[proof of \Extref{GK}{Theorem~4.8}]{GK} and Vaaler's
	approximation \cite[\Extref{Vaaler85}{Theorem~18}]{Vaaler85}.
	Classical exponent pairs cover $H\le T^s$ or $p\le42/67$.
	The remaining region is
	\[
	H>T^s,\qquad \frac{42}{67}\le p\le1-\beta,
	\]
	which we divide at $p=43/67$.
	
	\emph{The middle range: an explicit pair count.}
	For $42/67\le p\le43/67$, a van der Corput transformation gives
	the phase $\sqrt{Thq}-M\alpha q$, with
	$h\approx H$ and $U/100\le q\le2U$.
	Choose $J=\lceil T^{36/67}/M\rceil$.
	Let $B$ be the ordered-pair count in \eqref{eq:sampling-pairs},
	which retains only the fourth, fifth and sixth Taylor coefficients.
	Taylor expansion through degree six, the degree-five Vinogradov mean
	value theorem \cite[\Extref{BDG}{Theorem~1.1}]{BDG}, and the double large sieve give
	\[
	\frac{|S|}{H}\lesssim_\varepsilon T^\varepsilon
	\left[(T\rho^8(MJ)^6)^{1/30}
	\left(\frac B{HU}\right)^{1/30}
	+\frac{MJ}{\sqrt\rho}+1\right]
	\]
	by \Lemref{lem:pair-reduction}.
	The key estimate is $B\lesssim_\varepsilon HU T^\varepsilon$.
	\Secref{c:height} proves it by solving the coefficient constraints
	for the difference of the two $h$ variables. The linear error cancels,
	leaving a quadratic error small enough for Fourier expansion.
	A further van der Corput transformation factors the phase into two
	functions, each depending on a different pair of integer variables.
	The Robert--Sargos estimates \cite{RS06} then control both spacing counts.
	This yields \Propref{c:count-prop}. Since $MJ\eqsim T^{36/67}$,
	\[
	(T\rho^8(MJ)^6)^{1/30}
	\lesssim T^{(1+8p+6(36/67))/30},\qquad
	\frac{MJ}{\sqrt\rho}\lesssim T^{15/67}.
	\]
	The cancellation makes this shift length compatible with a pair
	count of diagonal size.
	
	\emph{The larger range: a distributional first-spacing estimate.}
	For $43/67\le p\le1-\beta$, \Secref{sec:secondspacing} applies
	a Farey dissection, two van der Corput transformations, and Watt's
	second-spacing bound \cite[\Extref{Watt04}{Lemma~8.2}]{Watt04}.
	The new input is \Thmref{thm:first}, proved in
	\Secref{sec:spacing} from Guth--Maldague's amplitude-dependent
	estimate \cite[\Extref{GM}{Theorem~2}]{GM} and an integer count using exact equalities
	in two Fourier coordinates.
	Small angular scales yield the fourth-power term in \eqref{eq:mixed},
	while the amplitude-dependent selection of larger scales yields the
	sixth-power term.
	Keeping these terms separate in the distributional double large sieve
	of \Appref{sec:sieve} gives \Lemref{lem:farey-Q}.
	The resulting denominator powers are summable, and the interval lengths
	in \eqref{eq:interval-lengths} give
	$|S|/H\lesssim_\varepsilon T^{\beta+\varepsilon}$.
	
	At $p=43/67$, the middle-range exponent and the first exponent in
	\eqref{eq:larger-leading} agree:
	\[
	\frac{1+8(43/67)+6(36/67)}{30}
	=\frac{10-43/67}{30}
	=\frac{209}{670}<\frac5{16}.
	\]
	The auxiliary summation and counting results are collected in
	\Appref{sec:transforms}. All reductions retain bounded-variation
	weights and arbitrary interval endpoints, as required to recover $E(X)$.
	
	\section{The main theorem: reduction to two estimates}\label{sec:main-proof}
	We first apply the classical circle-problem reduction, then isolate
	the two ranges not covered by classical exponent pairs.
	
	\subsection{Reduction to reciprocal sums}\label{sec:reduction}
	\begin{proposition}\label{prop:transfer}
		\Thmref{thm:recip} implies \Thmref{thm:main}.
	\end{proposition}
	\begin{proof}
		For $t\in\R$, set $\psi(t)=t-\lfloor t\rfloor-1/2$, and fix
		a nonintegral $X\ge4$.
		The classical reduction \cite[proof of \Extref{GK}{Theorem~4.8}]{GK}, followed by
		dyadic subdivision, expresses $E(X)$, up to $O(1)$, as a linear
		combination with coefficients $\pm4$ of $O(\log X)$ sums
		\[
		\sum_{n\in I\cap\Z}\psi\!\left(\frac XM F(n/M)\right),
		\qquad I\subseteq[M,2M],\quad 1\le M\le\sqrt X,
		\]
		where $I$ is an interval and $F$ is as in \eqref{eq:S}, with
		\[
		(\alpha,\kappa)\in
		\left\{\left(\pm\frac1{4M},0\right),
		\left(0,\pm\frac M{4X}\right)\right\}.
		\]
		Thus the four arguments are exactly $X/(4n\pm1)$ and
		$X/(4n)\pm1/4$. These are the four residue-class and sawtooth phases
		in the cited circle reduction. This identifies that standard reduction
		with the reciprocal family \eqref{eq:S} without adding a new analytic
		input. The pointwise Vaaler majorant applies to the convention
		$\psi(t)=\{t\}-1/2$ used here, including integral arguments.
		For any integer $J\ge2$, Vaaler's approximation
		\cite[\Extref{Vaaler85}{Theorem~18}]{Vaaler85} and partial summation in $h$ give
		\begin{equation}\label{eq:Vaaler}
			\begin{split}
				&\left|\sum_{n\in I\cap\Z}\psi\!\left(\frac XM F(n/M)\right)\right|\\
				&\quad\lesssim\frac MJ+
				\sum_{\substack{H=2^j<J\\j\ge0}}\frac1H
				\sup_{H\le H'\le\min(2H,J)}
				\left|\sum_{H\le h<H'}\sum_{n\in I\cap\Z}
				e\!\left(\frac{hX}{M}F(n/M)\right)\right|.
			\end{split}
		\end{equation}
		For $M\ge2X^\beta$, choose $J=\lfloor MX^{-\beta}\rfloor$.
		Then $J\ge2$, $M/J\lesssim X^\beta$, and every $H<J$ satisfies
		$H\le MX^{-\beta}$. \Thmref{thm:recip}, with $T=X$ and
		$g=G=1$, bounds each supremum in \eqref{eq:Vaaler} by
		$O_\varepsilon(HX^{\beta+\varepsilon})$. After multiplication by
		$H^{-1}$, each dyadic $h$-block therefore contributes
		$O_\varepsilon(X^{\beta+\varepsilon})$. There are only
		$O(\log X)$ such blocks. The dyadic $n$-blocks with
		$M<2X^\beta$ have total length
		$O(\sum_{M<2X^\beta}M)=O(X^\beta)$, so their contribution is
		trivial. Absorbing the logarithm into $X^\varepsilon$ proves the claim
		for nonintegral $X$. Finally, the counting function in the definition of
		$E(X)$ is right-continuous, hence so is $E(X)$. Letting a nonintegral
		argument decrease to an integral $X$ gives the same estimate there.
	\end{proof}
	
	\subsection{The two remaining ranges}
	\begin{proposition}[Classical bounds]\label{prop:preliminary}
		For the reciprocal sum \eqref{eq:S}, uniformly on subintervals,
		\begin{equation}\label{eq:preliminary}
			|S|/H\lesssim (HT)^{2/7}+\frac{M^2}{HT},\qquad
			|S|/H\lesssim \rho^{13/31}M^{3/31}+\frac{M^2}{HT}.
		\end{equation}
		Consequently $|S|/H\lesssim T^\beta$ if either $H\le T^s$ or
		$p\le42/67$.
	\end{proposition}
	
	\begin{proof}
		The exponent pairs $(k,l)=(2/7,4/7)$ and $(13/31,16/31)$
		\cite[Chapter~3, (3.3.4)]{GK} give, for every integer $h\approx H$
		and every interval $I\subseteq[M,2M]$,
		\[
		\left|\sum_{n\in I\cap\Z}e\!\left(\frac{hT}{4(n+M\alpha)}\right)\right|
		\lesssim (hT/M^2)^kM^l+\frac{M^2}{hT}.
		\]
		The constants are uniform in $\alpha$. Put
		$u=n+\lfloor M\alpha\rfloor$ and $a=\{M\alpha\}$. Then
		$n+M\alpha=u+a$, with $u\in\Z$, $u\eqsim M$ and $0\le a<1$.
		For real $x\eqsim M$ and integers $j\ge0$,
		\[
		\frac{d^j}{dx^j}(x+a)^{-2}
		=(-1)^j(j+1)!x^{-j-2}\bigl(1+O_j(M^{-1})\bigr).
		\]
		When $hT/M^2$ is bounded, the second-derivative test
		\cite[\Extref{GK}{Theorem~2.2}]{GK} gives the stronger bound $O(\sqrt M)$.
		If $M$ is bounded, both estimates in \eqref{eq:preliminary} follow
		from the trivial bound after enlarging the absolute constant, so we may
		assume $M$ is larger than a fixed constant. Since
		$\|G\|_\infty+\Var(G)\le1$, discrete Abel summation reduces the
		weighted $n$-sum to interval sums of the form just estimated, with no
		loss beyond an absolute constant. The same argument applies to an
		intersected dyadic interval. Finally, $|g|\le1$ and $h\approx H$, so
		summing these bounds over the $O(H)$ admissible integers $h$ and then
		dividing by $H$ gives \eqref{eq:preliminary}. To finish the parameter check,
		\[
		\frac{M^2}{HT}\le1,\qquad \frac{2(1+s)}7=\beta,\qquad
		\frac{13p+3m}{31}\le\frac{1293}{4154}<\beta
		\quad(p\le42/67).
		\]
	\end{proof}
	
	\begin{proposition}[The middle range]\label{prop:middle}
		Suppose
		\begin{equation}\label{c:middle-region}
			T^{42/67}\le\rho\le T^{43/67},\qquad H\ge T^s,\qquad M\le T^{1/2}.
		\end{equation}
		Then, uniformly in the shifts, weights, and interval endpoints,
		\begin{equation}\label{c:middle-result}
			|S|/H\lesssim_\varepsilon T^{\beta+\varepsilon}.
		\end{equation}
	\end{proposition}
	
	\begin{proposition}[The larger-\texorpdfstring{$\rho$}{rho} range]\label{prop:second}
		Suppose
		\begin{equation}\label{eq:parameter-region}
			\beta\le\tau\le24/67,\qquad \tau+s\le m\le1/2.
		\end{equation}
		Then, uniformly in the shifts, weights, and interval endpoints,
		\begin{equation}\label{eq:larger-result}
			|S|/H\lesssim_\varepsilon T^{\beta+\varepsilon}.
		\end{equation}
	\end{proposition}
	
	\begin{proof}[Proof of \Thmref{thm:recip}]
		If $H\le T^s$ or $p\le42/67$, \Propref{prop:preliminary}
		gives the conclusion. Otherwise $m>\tau+s$, $\tau\ge\beta$, and
		$\tau<25/67$. For $24/67\le\tau<25/67$, apply
		\Propref{prop:middle}. For $\beta\le\tau\le24/67$, apply
		\Propref{prop:second}. These cases cover \Thmref{thm:recip}.
		Use smaller positive losses in the supporting estimates so that all
		powers and logarithmic factors together are bounded by $T^\varepsilon$.
		\Propref{prop:transfer} then proves \Thmref{thm:main}.
	\end{proof}
	
	\section{The middle range: reduction to a pair count}\label{sec:sampling}
	Throughout Sections~\ref{sec:sampling}--\ref{c:height}, assume
	\eqref{c:middle-region} and take $T$ sufficiently large. Fix
	\begin{equation}\label{c:middle-shift}
		J=\left\lceil\frac{T^{36/67}}M\right\rceil,\qquad \omega=M\alpha.
	\end{equation}
	The van der Corput transformation and the degree-five Vinogradov
	mean value theorem reduce the estimate to the following pair count.
	
	\subsection{The sufficient estimates}
	For $i\in\{1,\ldots,6\}$, use
	$\binom{1/2}{i}=\frac1{i!}\prod_{j=0}^{i-1}(\frac12-j)$.
	For $h,q>0$, define
	\begin{equation}\label{eq:Taylor-coefficients}
		A_i(h,q)=\binom{1/2}{i}\sqrt{Th}\,q^{1/2-i}
		-\omega\mathbf1_{\{i=1\}}.
	\end{equation}
	Fix a sufficiently large constant $C_0\ge1$ and put
	\begin{equation}\label{eq:sampling-pairs}
		\begin{gathered}
			\Omega=([H,2H)\cap\Z)\times([U/100,2U]\cap\Z),\\
			B=\#\left\{((h,q),(h',q'))\in\Omega^2:
			\begin{array}{l}
				\|A_4(h',q')-A_4(h,q)\|\le C_0J^{-4},\\
				|A_i(h',q')-A_i(h,q)|\le C_0J^{-i}\quad(i=5,6)
			\end{array}\right\}.
		\end{gathered}
	\end{equation}
	We retain $J$, $\omega$, $A_i$, $\Omega$, $B$ and $C_0$
	in \Secref{c:height}.
	
	\begin{lemma}\label{lem:pair-reduction}
		For the choice \eqref{c:middle-shift},
		\begin{equation}\label{eq:pair-reduction}
			|S|/H\lesssim_\varepsilon T^\varepsilon
			\left[(T\rho^8(MJ)^6)^{1/30}
			\left(\frac B{HU}\right)^{1/30}
			+\frac{MJ}{\sqrt\rho}+1\right].
		\end{equation}
	\end{lemma}
	
	\begin{proposition}\label{c:count-prop}
		For the choice \eqref{c:middle-shift},
		\begin{equation}\label{c:count}
			B\lesssim_\varepsilon HU T^\varepsilon.
		\end{equation}
	\end{proposition}
	
	\begin{proof}[Proof of \Propref{prop:middle}]\label{c:middle-proof}
		The preceding two estimates give the conclusion, since
		\begin{equation}\label{c:middle-exponent}
			\begin{gathered}
				T^{36/67}\le MJ\le2T^{36/67},\\
				(T\rho^8(MJ)^6)^{1/30}
				\lesssim T^{(1+8\cdot43/67+6\cdot36/67)/30}=T^{209/670},\\
				\frac{MJ}{\sqrt\rho}\le2T^{15/67}<2T^\beta.
			\end{gathered}
		\end{equation}
		Use smaller positive losses in the two estimates.
	\end{proof}
	
	\subsection{Proof of the mean value reduction}
	\begin{proof}[Proof of \Lemref{lem:pair-reduction}]
		For $h,q>0$, set $f_h(q)=\sqrt{Thq}-\omega q$. For each integer
		$h\approx H$, let $c_h\in\mathbb C$ and let $I_h\subseteq[U/100,2U]$
		be an interval. Define
		\begin{equation}\label{eq:sqrt-phase}
			\mathcal S=\sup_{|c_h|\le1,\ I_h\subseteq[U/100,2U]}
			\left|\sum_{h\approx H}c_h\sum_{q\in I_h\cap\Z}e(f_h(q))\right|,
		\end{equation}
		where the supremum is over these coefficient and interval families.
		Abel summation and the van der Corput
		formula \cite[\Extref{Huxley96}{Lemma~5.5.3}]{Huxley96} give
		\begin{equation}\label{eq:reciprocal-transfer}
			\frac{|S|}{H}\lesssim\frac{M}{H\sqrt\rho}\mathcal S
			+\log(2+U)+\frac{M}{\sqrt\rho}.
		\end{equation}
		Here the inversion formula is applied on the actual subinterval of the
		$n$-sum. Its endpoint-uniform error is the displayed
		$O(M/\sqrt\rho+\log(2+U))$. The derivative check below gives
		$|\phi_h''|\eqsim\rho/M^2$ and a dual frequency interval of length
		$O(U)$, which are precisely the scales entering that error term. If an
		endpoint derivative is integral, changing the half-weight convention
		costs $O(M/\sqrt\rho)$ and is already absorbed by the same error.
		For $h>0$ and $x>-M\alpha$, define
		$\phi_h(x)=hT/[4(x+M\alpha)]$. For $q>0$, let $x_{h,q}$ be the
		unique solution of $\phi_h'(x_{h,q})=-q$. Then
		\begin{equation}\label{eq:reciprocal-stationary}
			\begin{gathered}
				x_{h,q}=\frac{\sqrt{hT}}{2\sqrt q}-M\alpha,\qquad
				\phi_h(x_{h,q})+qx_{h,q}=f_h(q),\\
				(\phi_h''(x_{h,q}))^{-1/2}
				=\frac{(hT)^{1/4}}{2q^{3/4}}
				=\frac{M}{2\sqrt\rho}(h/H)^{1/4}(q/U)^{-3/4}.
			\end{gathered}
		\end{equation}
		For $h\approx H$ and $x\in[M,2M]$, these frequencies lie in $[U/100,2U]$.
		$\phi_h''\eqsim\rho/M^2$ and
		$|\phi_h^{(j)}|\lesssim\rho M^{-j}$ for $j=3,4$.
		The factor $(q/U)^{-3/4}$ has uniformly bounded variation on
		$[U/100,2U]$. Apply
		\eqref{eq:abel-bv} to it and to $G$. The remaining factors have bounded
		modulus and are absorbed in $c_h$. This proves
		\eqref{eq:reciprocal-transfer}, uniformly on the stated intervals.
		
		We next prove
		\begin{equation}\label{eq:average-mixed}
			\mathcal S\lesssim_\varepsilon T^\varepsilon
			\left[HJ+\left(\frac{\rho J^6}{U^6}\right)^{1/30}
			(HU)^{14/15}B^{1/30}\right].
		\end{equation}
		The choice of $J$ and \eqref{c:middle-region} imply
		\begin{equation}\label{eq:shift-checks}
			\begin{gathered}
				J\ge T^{5/134},\qquad J/U=(MJ)/\rho\lesssim T^{-6/67},\\
				\rho(J/U)^7=(MJ)^7/\rho^6\lesssim1,\qquad
				\rho J^6/U^6=(MJ)^6/\rho^5\ge T^{1/67},\\
				\rho/U^5=M^5/\rho^4\le T^{-1/134}.
			\end{gathered}
		\end{equation}
		Averaging over integer shifts $1\le r\le J$ changes each interval sum by
		$O(J)$. Define the Taylor remainder $R_{h,q}$ by
		\[
		f_h(q+r)=f_h(q)+\sum_{i=1}^6A_i(h,q)r^i+R_{h,q}(r)
		\qquad(0\le r\le J).
		\]
		Taylor's formula gives
		\[
		\sup_{\substack{(h,q)\in\Omega\\0\le r\le J}}|\partial_rR_{h,q}(r)|
		\lesssim\frac{\rho J^6}{U^7}\lesssim J^{-1}.
		\]
		For $b_{h,q}(r)=e(R_{h,q}(r))$, therefore,
		\[
		\sup_{(h,q)\in\Omega}|b_{h,q}(J)|+
		\sum_{r=1}^{J-1}\sup_{(h,q)\in\Omega}|b_{h,q}(r+1)-b_{h,q}(r)|\lesssim1.
		\]
		By \eqref{eq:abel-bv}, it suffices to bound, for every integer
		$1\le j\le J$ and every complex family $(d_{h,q})_{(h,q)\in\Omega}$
		with $|d_{h,q}|\le1$,
		\[
		\sum_{(h,q)\in\Omega}d_{h,q}
		\sum_{r=1}^j e\!\left(\sum_{i=1}^6A_i(h,q)r^i\right).
		\]
		For complex coefficients $d_1,\ldots,d_J$ and
		$y=(y_1,\ldots,y_5)\in\T^5$, the degree-five Vinogradov mean value
		theorem \cite[\Extref{BDG}{Theorem~1.1}]{BDG}, by expansion and orthogonality, gives
		\begin{equation}\label{eq:VMVT}
			\sup_{|d_r|\le1}\int_{\T^5}
			\left|\sum_{r=1}^Jd_r e\!\left(\sum_{i=1}^5y_ir^i\right)\right|^{30}dy
			\lesssim_\varepsilon J^{15+\varepsilon}.
		\end{equation}
		Indeed, expansion of the left side leaves only the degree-five Vinogradov
		system. Every surviving coefficient product has modulus at most one, so
		Theorem~1.1 of \cite{BDG} with $s=15$ and degree $5$ bounds the integral
		by $J_{15,5}(J)\lesssim_\varepsilon J^{15+\varepsilon}$.
		For a sufficiently large fixed $C\ge1$, apply
		\eqref{eq:elementary-sieve} with moment exponent $30$ and
		\[
		D_i=1/2\ (1\le i\le5),\qquad D_6=C\rho/U^6,\qquad
		E_i=J^i\ (1\le i\le6).
		\]
		For $(h,q)\in\Omega$, write $A_i=A_i(h,q)$. Use
		$A_i-\lfloor A_i+1/2\rfloor$ for the first five sieve coordinates
		and $A_6$ for the sixth.
		The sieve pair count is bounded by $B$: discard its first three
		restrictions and use $|A_5|\lesssim\rho/U^5=o(1)$, so that the
		fifth centered representative equals $A_5$.
		For $x=(x_1,\ldots,x_6)\in\prod_{i=1}^6[-D_i,D_i]$, fix $x_6$
		and absorb its phase into $d_r$. Equation~\eqref{eq:VMVT} then gives
		\[
		\begin{gathered}
			P:=\prod_{i=1}^6(1+D_iE_i)\eqsim J^{15}\frac{\rho J^6}{U^6},\\
			\sup_{|d_r|\le1}\left\|\sum_{r=1}^j d_r
			e\!\left(\sum_{i=1}^6x_ir^i\right)\right\|_{L^{30}_\#(\prod_{i=1}^6[-D_i,D_i])}
			\lesssim_\varepsilon J^{1/2+\varepsilon}.
		\end{gathered}
		\]
		The sieve bound, divided by $J$, is the second term of
		\eqref{eq:average-mixed}. Abel summation above and the $O(HJ)$
		endpoint error prove that estimate. Finally, the thirtieth power of
		the resulting main term in \eqref{eq:reciprocal-transfer} is
		\[
		\left[\frac{M}{H\sqrt\rho}
		\left(\frac{\rho J^6}{U^6}\right)^{1/30}
		(HU)^{14/15}B^{1/30}\right]^{30}
		=T\rho^8(MJ)^6\frac B{HU},
		\]
		using $U=\rho/M$ and $\rho^2=THU$.
		The other terms are bounded by those in \eqref{eq:pair-reduction}.
	\end{proof}
	
	\section{The pair count: an explicit formula}\label{c:height}
	We prove \Propref{c:count-prop}. For a pair
	$((h,q),(h',q'))$ counted in \eqref{eq:sampling-pairs}, we solve
	$\|A_4(h',q')-A_4(h,q)\|\le C_0J^{-4}$ for $h'-h$. Set
	\begin{equation}\label{eq:ab}
		a=\frac{U^6}{\rho J^6}=\frac{\rho^5}{(MJ)^6},\qquad
		b=\frac{U^5}{\rho J^5}=\frac{\rho^4}{(MJ)^5},\qquad a=\frac{Ub}{J}.
	\end{equation}
	We use the following consequences of \eqref{c:middle-region} and
	$MJ\ge T^{36/67}$:
	\begin{equation}\label{c:smallness}
		\begin{gathered}
			a\le T^{-1/67},\qquad UJ^{-5}=\rho b/U^4\le T^{-3/67},\qquad
			Jb\lesssim T^{-43/1340},\\
			HJb^2=\frac{\rho^9}{T(MJ)^9}\le T^{-4/67},\qquad
			\frac{H\rho b^3}{U^3}=\frac{\rho^{11}M^4}{T(MJ)^{15}}\le1,\\
			\frac{\rho b^2}{U^3}=\frac{\rho^6M^3}{(MJ)^{10}}\le T^{-3/134},\qquad
			\frac{\rho b}{HU^3}=\frac{TM^2\rho}{(MJ)^5}\le T^{-3/67}.
		\end{gathered}
	\end{equation}
	For $Jb$, also use $M\ge T^{1+s}/\rho$. Balancing
	$(MJ)^{15}=T\rho^{11}$ at $\rho=T^{43/67}$ gives $MJ=T^{36/67}$,
	explaining \eqref{c:middle-shift}.
	
	\subsection{Localization and the explicit formula}
	For a counted pair, write $A_i=A_i(h,q)$ and $A_i'=A_i(h',q')$
	for $1\le i\le6$, and set $k=q'-q$, $w=h'-h$. Only these
	primed coefficients denote evaluation at the second point. Primes on
	functions below denote derivatives.
	The identities
	\begin{equation}\label{c:identities}
		\begin{gathered}
			A_4=-\frac5{128}\sqrt{Th}\,q^{-7/2},\quad
			A_5=\frac7{256}\sqrt{Th}\,q^{-9/2},\quad
			A_6=-\frac{21}{1024}\sqrt{Th}\,q^{-11/2},\\
			A_5=-\frac43qA_6,\qquad A_4=-\frac{10}{7}qA_5
		\end{gathered}
	\end{equation}
	give
	\begin{equation}\label{c:localization}
		|k|\lesssim Ua,\qquad h+w=h(1+k/q)^9+O(Hb),\qquad
		A_4'-A_4=A_4k/q+O(UJ^{-5}).
	\end{equation}
	Indeed,
	\[
	\begin{aligned}
		|k||A_6|&\lesssim J^{-5}+UJ^{-6},& |A_6|&\eqsim\rho/U^6,\\
		\left(\frac{A_5'}{A_5}\right)^2
		&=\frac{h+w}{h}\left(\frac q{q+k}\right)^9=1+O(b),\\
		A_4'-A_4&=-\frac{10}{7}\bigl(kA_5+q'(A_5'-A_5)\bigr)
		=A_4k/q+O(UJ^{-5}).
	\end{aligned}
	\]
	Since $C_0J^{-4}<1/2$ for large $T$, there is a unique
	$\ell\in\Z$ with $|A_4'-A_4-\ell|\le C_0J^{-4}$.
	Let $B^0$ count the pairs for which $\ell=0$.
	In this case the fourth and fifth restrictions imply
	\[
	|k|\lesssim Ub,\quad |w|\lesssim Hb,\quad
	h'=h(1+k/q)^7+O(HJb/U),\quad
	|qw-7hk|\lesssim HJb+HUb^2.
	\]
	The fourth restriction gives
	\[
	k=0\ \Longrightarrow\ |w|\lesssim HJb/U\le Jb=o(1),\qquad
	w=0\ \Longrightarrow\ |k|\lesssim Jb=o(1).
	\]
	These pairs are diagonal. Otherwise fix $h,k$ and the integer
	$D_0=qw-7hk$.
	The divisor bound \cite[\Extref{HW}{Theorem~315}]{HW} applied to $qw=7hk+D_0\ne0$
	leaves $O_\varepsilon(T^\varepsilon)$ choices of $(q,w)$. Hence
	\[
	B^0\lesssim_\varepsilon HU+T^\varepsilon H(Ub)(1+HJb+HUb^2),
	\]
	and
	\begin{equation}\label{c:zero-count}
		\frac{B^0}{HU}\lesssim_\varepsilon
		T^\varepsilon(1+b+HJb^2+HUb^3)\lesssim_\varepsilon T^\varepsilon,
	\end{equation}
	since $HUb^3=aHJb^2$ and $b=aJ/U=o(1)$.
	
	For $\ell\ne0$, \eqref{c:localization} and $UJ^{-5}=o(1)$ give
	$|k|\rho/U^5\gtrsim1$. Split $|k|/U$ into dyadic ranges $[r,2r)$ with $r=2^j$, $j\in\Z$.
	Interchanging the two points if necessary, we have
	\begin{equation}\label{c:r-range}
		k\eqsim Ur>0,\qquad w\eqsim Hr>0,\qquad
		U^4/\rho\lesssim r\lesssim a,\qquad b/r=o(1).
	\end{equation}
	Here $w=9hk/q\,(1+O(a))+O(Hb)$. Put
	$n=-\ell\in\N$, $K=Ur$ and $Z=\rho r/U^4$.
	Then $n\eqsim Z$ and $K,Z,Hr\gtrsim1$ on every nonempty range.
	Set $\delta=(q/q')(1+w/h)^{1/9}-1$. Since
	$A_5'/A_5=(1+\delta)^{9/2}$, one has $|\delta|\lesssim b$ and
	\[
	\begin{gathered}
		A_4'-A_4=\frac{A_4}{q}\big(q'(1+\delta)^{9/2}-q\big),\\
		(h+w)^{1/9}-h^{1/9}
		=\frac{h^{1/9}k}{q}\left(1+\frac{q'}k\delta\right),\qquad
		q=\frac{k}{(1+w/h)^{1/9}-1}+O(Ub/r).
	\end{gathered}
	\]
	The constant and linear terms cancel in
	\[
	q'(1+\delta)^{9/2}-q-k\left(1+\frac{q'}k\delta\right)^{9/2}
	=O\big((q'+q'^2/k)b^2\big)=O(Ub^2/r),
	\]
	where $|q'\delta/k|\lesssim b/r=o(1)$.
	As $|A_4|/q\eqsim\rho/U^5$, the fourth restriction yields
	\begin{equation}\label{c:root-error}
		\left|\frac{5\sqrt T}{128k^{7/2}}
		\big((h+w)^{1/9}-h^{1/9}\big)^{9/2}-n\right|
		\lesssim J^{-4}+\frac{\rho b^2}{U^4r}
		\lesssim\frac{\rho b^2}{U^4r}.
	\end{equation}
	The last inequality uses $J^{-4}=(r/a)\rho b^2/(U^4r)$.
	For fixed $k,n\in\N$, define $\mathcal A_{k,n},D_r>0$ and
	$f_{k,n}:(0,\infty)\to\R$ by
	\begin{equation}\label{c:inverse}
		\mathcal A_{k,n}=\frac{16384n^2k^7}{25T},\qquad
		f_{k,n}(h)=(h^{1/9}+\mathcal A_{k,n}^{1/9})^9-h,\qquad
		D_r=\frac{Hb^2}{r}.
	\end{equation}
	Since $n\eqsim Z$, the relative error in \eqref{c:root-error} is
	$O((b/r)^2)=o(1)$. Taking the $2/9$ power and then the ninth power gives
	\[
	|(h+w)^{1/9}-h^{1/9}-\mathcal A_{k,n}^{1/9}|
	\lesssim H^{1/9}b^2/r,\qquad |w-f_{k,n}(h)|\lesssim D_r.
	\]
	Moreover, $\rho^2=THU$ gives
	\[
	\mathcal A_{k,n}\eqsim Z^2K^7/T
	=\rho^2r^9/(TU)=Hr^9,
	\]
	and
	$f_{k,n}(h)\eqsim Hr$ for $h\approx H$, uniformly on fixed rectangular
	enlargements of the $n,k$ ranges.
	
	For this fixed $r$, choose intervals $I_1\subseteq[cZ,CZ]$ and
	$I_2\subseteq[cK,CK]$ containing the ranges of $n$ and $k$, respectively.
	The intervals do not depend on the summation indices. Fix $C_1\ge1$
	so that the bound $|w-f_{k,n}(h)|\le C_1D_r$ holds for every counted pair.
	For each $n,k,h,w$, the displayed formula for $q$ leaves
	$O(Ub/r)$ integers $q$, since $Ub/r\gtrsim J$.
	Let $B_r$ count the pairs with $\ell\ne0$ and
	$rU\le q'-q<2rU$. Then
	\begin{equation}\label{c:common-factor}
		\frac{B_r}{HU}\lesssim\frac b{Hr}
		\sum_{n\in I_1}\sum_{k\in I_2}\sum_{h\approx H}
		\#\{w\in\Z:|w-f_{k,n}(h)|\le C_1D_r\},\qquad
		B\lesssim B^0+\sum_r B_r.
	\end{equation}
	The sum in $r$ has $O(\log T)$ terms, since $Ur\gtrsim1$ and $r\lesssim a$.
	
	\subsection{Fourier expansion and the classical spacing bounds}
	Choose a fixed $c>0$ with $cC_1\le1/8$, and set
	\begin{equation}\label{c:fourier-cutoff}
		Q=\left\lfloor c\frac{ZKb}{r}\right\rfloor
		=\left\lfloor c\frac{\rho br}{U^3}\right\rfloor.
	\end{equation}
	If $Q<2$, \eqref{c:common-factor} and the trivial bound $O(1+D_r)$
	for its innermost count give
	\[
	\frac{B_r}{HU}\lesssim\frac{\rho br}{U^3}
	+\frac{H\rho b^3}{U^3}\lesssim1.
	\]
	Suppose $Q\ge2$. Then $Q\eqsim ZKb/r$ and
	\begin{equation}\label{c:cutoff-checks}
		QD_r\le c\frac{H\rho b^3}{U^3}\le c,\qquad
		\frac Q{Hr}\lesssim\frac{\rho b}{HU^3}\le T^{-3/67}.
	\end{equation}
	Hence $QC_1D_r\le1/8$.
	For real $u$ with $\|u\|\le C_1D_r$, the Fej\'er kernel gives
	\[
	\mathbf1_{\{\|u\|\le C_1D_r\}}
	\le\frac2{Q^2}\left|\sum_{j=0}^{Q-1}e(ju)\right|^2.
	\]
	For integers $1\le v<Q$, put
	\[
	\Sigma_v=\sum_{n\in I_1}\sum_{k\in I_2}\sum_{h\approx H}e(vf_{k,n}(h)).
	\]
	Each triple $n,k,h$ admits at most one $w$, since $2C_1D_r<1$.
	Expand the square before taking absolute values in $n,k$. Using
	$Q\eqsim ZKb/r$ in \eqref{c:common-factor} gives
	\begin{equation}\label{c:fejer-count}
		\frac{B_r}{HU}\lesssim1+\sum_{1\le v<Q}\frac{|\Sigma_v|}{ZKH}.
	\end{equation}
	
	Fix $V=2^j<Q$ with an integer $j\ge0$, and put $Y=HrV$.
	Throughout the following transformation, $v$ is an integer satisfying
	$V\le v<\min(2V,Q)$, $n\in I_1\cap\N$ and $k\in I_2\cap\N$.
	Since $Q\ge2$, \eqref{c:cutoff-checks} gives $Y\ge Hr\gtrsim T^{3/67}$.
	For $f=f_{k,n}$ and $H\le h\le2H$, \eqref{c:inverse} gives
	\[
	f(h)=\sum_{\nu=1}^9\binom9\nu
	\mathcal A_{k,n}^{\nu/9}h^{1-\nu/9},\qquad
	\mathcal A_{k,n}\eqsim Hr^9.
	\]
	Differentiating this identity, and using $r=o(1)$, yields
	\begin{equation}\label{c:inverse-derivatives}
		\begin{gathered}
			f'(h)=\big(1+(\mathcal A_{k,n}/h)^{1/9}\big)^8-1\eqsim r,\\
			f''(h)=-\frac{8\mathcal A_{k,n}^{1/9}}{9h^{10/9}}
			\big(1+(\mathcal A_{k,n}/h)^{1/9}\big)^7,\qquad |f''(h)|\eqsim r/H,\\
			|f^{(j)}(h)|\lesssim H^{1-j}\sum_{\nu=1}^9r^\nu
			\lesssim rH^{1-j}\quad(j=3,4).
		\end{gathered}
	\end{equation}
	For real $v,t>0$, define
	\begin{equation}\label{c:dz}
		D(v,t)=\big(v^{-1/8}-(v+t)^{-1/8}\big)^{-8}.
	\end{equation}
	For fixed $k,n$, the first formula below defines $h_*>0$ for every
	$v,t>0$. It satisfies $vf'(h_*)=t$, and direct substitution gives
	\begin{equation}\label{c:amplitude}
		\begin{gathered}
			h_*=\mathcal A_{k,n}\big((1+t/v)^{1/8}-1\big)^{-9}
			=-\mathcal A_{k,n}\partial_tD(v,t),\\
			vf(h_*)-th_*=\mathcal A_{k,n} D(v,t),\qquad
			(v|f''(h_*)|)^{-1/2}
			=\sqrt{\mathcal A_{k,n}}\sqrt{\partial_t^2D(v,t)}.
		\end{gathered}
	\end{equation}
	Since $v\eqsim V$ and $Y=HrV$,
	\[
	|(vf)''(h)|\eqsim Y/H^2,\qquad
	|(vf)^{(j)}(h)|\lesssim Y/H^j\quad(j=3,4).
	\]
	The van der Corput transformation \cite[\Extref{Huxley96}{Lemma~5.5.3}]{Huxley96}
	in $h$ therefore has error $O(H/\sqrt Y+\log(2+rV))$.
	The cited lemma is applied on each actual $h$-subinterval. Since
	$|(vf)''|\eqsim Y/H^2$ and the derivative range has length $O(rV)$,
	this is the uniform inversion error. A change in the half-weight of an
	endpoint contributes $O(H/\sqrt Y)$ and is already included.
	Its integer frequencies
	are restricted by $H\le h_*<2H$, and on this range
	\[
	t\eqsim rV,\qquad (v|f''(h_*)|)^{-1/2}\eqsim H/\sqrt Y.
	\]
	The phase and amplitude in \eqref{c:amplitude} each factor into a
	function of $(n,k)$ and a function of $(v,t)$.
	Choose fixed $0<c_2<C_2$ so that each stationary frequency satisfies
	$c_2rV\le t\le C_2rV$, and put
	\[
	\begin{gathered}
		\mathcal I=\{(v,t)\in\N^2:V\le v<\min(2V,Q),\ c_2rV\le t\le C_2rV\},\\
		x_{n,k}=(n/Z)^2(k/K)^7,\qquad y_{v,t}=\frac{r^8}{V}D(v,t).
	\end{gathered}
	\]
	The normalized coordinates lie in fixed positive intervals, and
	$\mathcal A_{k,n} D(v,t)=(16384/25)Yx_{n,k}y_{v,t}$.
	If $\mathcal I=\varnothing$, the transformed main sum is zero.
	Otherwise the restriction $H\le h_*<2H$ is
	$H/\mathcal A_{k,n}\le-\partial_tD(v,t)<2H/\mathcal A_{k,n}$,
	and \Lemref{lem:finite-separation} separates it at cost $O(\log T)$.
	For a sufficiently large fixed $C\ge1$, define the following counts,
	with all indices integer:
	\[
	\begin{aligned}
		B_x&=\#\{(n,k,n',k'):n,n'\in I_1,\ k,k'\in I_2,
		\ |x_{n,k}-x_{n',k'}|\le C/Y\},\\
		B_y&=\#\{((v,t),(v',t'))\in\mathcal I^2:
		|y_{v,t}-y_{v',t'}|\le C/Y\}.
	\end{aligned}
	\]
	We claim
	\begin{equation}\label{c:pair-counts}
		B_x\lesssim_\varepsilon T^\varepsilon
		\left(ZK+\frac{(ZK)^2}{Y}\right),\qquad
		B_y\lesssim_\varepsilon T^\varepsilon V^2.
	\end{equation}
	For the first bound, apply \cite[\Extref{RS06}{Section~6.1}, (6.4)]{RS06}
	with monomial exponents $2,7$ and tolerance $O(Y^{-1})$.
	A bounded dyadic subdivision and \eqref{eq:pair-comparison}
	cover the fixed intervals $I_1,I_2$.
	For the second, use
	\[
	v^{-1/8}-(v+t)^{-1/8}=rV^{-1/8}y_{v,t}^{-1/8}.
	\]
	A pair counted by $B_y$ therefore satisfies
	\[
	|v^{-1/8}-(v+t)^{-1/8}-v'^{-1/8}+(v'+t')^{-1/8}|
	\lesssim(r/Y)V^{-1/8}.
	\]
	For sufficiently large $T$, $t\le C_2rV\le V/8$, and hence
	\[
	(v,v+t)\in
	\begin{cases}
		[3V/4,3V/2]^2,&v\le11V/8,\\
		[5V/4,5V/2]^2,&v>11V/8.
	\end{cases}
	\]
	Apply \eqref{eq:pair-comparison} to the values
	$v^{-1/8}-(v+t)^{-1/8}$ and these two groups.
	Within each group the injective map
	\[
	(v,t,v',t')\longmapsto(v,v'+t',v+t,v')
	\]
	has all four entries in one interval $[A,2A]$, with $A\eqsim V$.
	Robert--Sargos \cite[\Extref{RS06}{Theorem~2}, \Extref{RS06}{Section~5}]{RS06}, with exponent $-1/8$
	and tolerance $O(r/Y)$, therefore gives
	\[
	B_y\lesssim_\varepsilon T^\varepsilon(V^2+rV^4/Y)
	=T^\varepsilon V^2(1+V/H)\lesssim T^\varepsilon V^2,
	\]
	since $V/H\le Q/H\lesssim rT^{-3/67}\lesssim1$ by
	\eqref{c:cutoff-checks} and $r\lesssim a=o(1)$.
	
	For integers $V\le v<\min(2V,Q)$, insert complex coefficients
	$c_v$ with $|c_v|\le1$ before the transformation.
	They affect only the $(v,t)$ factor. The double large sieve
	\cite[\Extref{RS06}{Lemma~8}, \Extref{RS06}{Section~6.1}]{RS06}, \eqref{c:amplitude}, and
	\eqref{c:pair-counts} give
	\begin{equation}\label{c:direct-sieve}
		\begin{split}
			\frac1V\sum_{\substack{v\approx V\\v<Q}}\frac{|\Sigma_v|}{ZKH}
			&\lesssim_\varepsilon
			\frac{T^\varepsilon}{VZK}\sqrt{B_xB_y}
			+Y^{-1/2}+H^{-1}\log(2+rV)\\
			&\lesssim_\varepsilon T^\varepsilon
			\left((ZK)^{-1/2}+(HrV)^{-1/2}\right).
		\end{split}
	\end{equation}
	For the absolute values choose $c_v=\overline{\Sigma_v}/|\Sigma_v|$
	when $\Sigma_v\ne0$, and zero otherwise. The transformation error for
	each $n,k,v$ is $O(H/\sqrt Y+\log(2+rV))$. Endpoint changes are included.
	The last inequality uses $rV\le H$ from \eqref{c:cutoff-checks}.
	All parameters are bounded by fixed powers of $T$.
	Finally, sum \eqref{c:direct-sieve} over dyadic $V<Q$ in
	\eqref{c:fejer-count}, using
	$\sum_{j\ge0,\,2^j<Q}2^j\lesssim Q$ and
	$\sum_{j\ge0,\,2^j<Q}2^{j/2}\lesssim\sqrt Q$:
	\begin{equation}\label{c:factorized-count}
		\begin{split}
			\frac{B_r}{HU}
			&\lesssim_\varepsilon T^\varepsilon
			\left(1+\frac Q{\sqrt{ZK}}+\sqrt{\frac Q{Hr}}\right)\\
			&\lesssim_\varepsilon T^\varepsilon
			\left(1+\sqrt{\frac{\rho b^2}{U^3}}
			+\sqrt{\frac{\rho b}{HU^3}}\right)
			\lesssim_\varepsilon T^\varepsilon.
		\end{split}
	\end{equation}
	Together with \eqref{c:zero-count} and the $O(\log T)$ ranges in
	\eqref{c:common-factor}, this proves \Propref{c:count-prop}.
	
	\section{The larger range: the second-spacing estimate}\label{sec:secondspacing}
	We prove \Propref{prop:second} from Watt's estimate
	\cite[\Extref{Watt04}{Lemma~8.2}]{Watt04} and \Thmref{thm:first}. The role of
	Watt's result is only to count nearby minor arcs after the parameter
	hypotheses are verified below. The first-spacing input is supplied separately
	by \Thmref{thm:first}, so the two ingredients remain logically distinct.
	Throughout this section,
	\[
	H=T^{m-\tau},\qquad M=T^m,\qquad
	\beta\le\tau\le24/67,\qquad \tau+s\le m\le1/2.
	\]
	
	\subsection{The required estimate and the choice of parameters}
	Put
	\begin{equation}\label{eq:interval-lengths}
		\begin{gathered}
			N_{\mathrm I}=2^{18/35}M^{57/35}H^{-22/35}T^{-17/35},\\
			N_{\mathrm {II}}=(M^6/(TH^5))^{1/7},\qquad
			N_{\mathrm {III}}=2(M^6/(T^2H))^{1/3},\\
			N=\begin{cases}
				\lceil N_{\mathrm I}\rceil,&HM^9\ge T^4,\\
				\lceil\min(N_{\mathrm {II}},N_{\mathrm {III}})\rceil,&HM^9<T^4.
			\end{cases}
		\end{gathered}
	\end{equation}
	These are \cite[(8.18)--(8.20)]{Watt04} with parameters
	$(2M,2H,N,2T)$.
	
	Let $j_0=\lceil2T/M^2\rceil$. For $x>-M\alpha$ and
	$z>-\alpha/2$, respectively, define
	\[
	\begin{gathered}
		f(x)=\frac{T}{4(x+M\alpha)}+\frac{T\kappa}{M}+j_0x,\\
		\Phi(z)=\frac1{16}\log(z+\alpha/2)+\frac\kappa2z
		+\frac{j_0M^2}{2T}z^2.
	\end{gathered}
	\]
	For integer $h,n$,
	\[
	e(hj_0n)=1,\qquad
	f(x)=\frac{2(2T)}{2M}\Phi'(x/(2M)),\qquad
	[M,2M]=[(2M)/2,2M].
	\]
	For $1/3\le z\le3$, set $u=z+\alpha/2$. Then
	\begin{equation}\label{eq:primitive-derivatives}
		\begin{gathered}
			5/24\le u\le25/8,\qquad 2\le j_0M^2/T\le3,\\
			\Phi''(z)=\frac{j_0M^2}{T}-\frac1{16u^2}\ge\frac{14}{25},\qquad
			\Phi^{(j)}(z)=\frac{(-1)^{j-1}(j-1)!}{16u^j}\quad(j=3,4,5).
		\end{gathered}
	\end{equation}
	Take Watt's auxiliary function $F$ to be $\Phi$. The preceding identity
	then says that Watt's function $f(x,0)$ is exactly the $f(x)$ used here,
	so his defining relation (2.1) is satisfied by construction. By
	\eqref{eq:primitive-derivatives}, choose fixed positive constants
	$C_2,\ldots,C_5$ satisfying \cite[(2.2)--(2.5)]{Watt04} for this
	family, and define $C_R=4C_3$. In Watt's notation the auxiliary modulus
	is one. Moreover, $M\le T^{1/2}$ gives the required lower bound
	$P/Q\eqsim(2T)/(2M)^2\ge1/2$ after the parameter substitution below.
	We also choose $c_*$ no larger than the fixed constant in Watt's
	condition~(5.2).
	Now set
	\begin{equation}\label{eq:NR}
		R=\left\lceil\sqrt{C_RM^3/(2NT)}\right\rceil,\qquad W=H/R.
	\end{equation}
	\Secref{sec:farey-inputs} verifies that
	$R\eqsim(M^3/(NT))^{1/2}$ and $N,R,W\to\infty$.
	
	\begin{lemma}\label{lem:farey-Q}
		With these choices,
		\begin{equation}\label{eq:second}
			\frac{|S|}{H}\lesssim_\varepsilon T^\varepsilon
			\left[\frac{M}{\sqrt{HN}}W^{2/9}
			+\frac{M}{N^{3/4}H^{1/4}}W^{1/3}\right].
		\end{equation}
	\end{lemma}
	\begin{proof}[Proof of \Propref{prop:second}]
		Substituting \eqref{eq:interval-lengths} and \eqref{eq:NR} gives
		\begin{equation}\label{eq:larger-leading}
			\begin{aligned}
				\frac{M}{\sqrt{HN}}W^{2/9}
				&\eqsim T^{1/9}M^{2/3}H^{-5/18}N^{-7/18}\\
				&\lesssim
				\begin{cases}
					T^{(9+\tau)/30},&HM^9\ge T^4,\\
					T^{(2m+1)/6}+T^{(10-7m+4\tau)/27},&HM^9<T^4.
				\end{cases}
			\end{aligned}
		\end{equation}
		The three exponents satisfy
		\[
		\begin{aligned}
			\frac{9+\tau}{30}&\le\frac{9+24/67}{30}=\beta,\\
			10m-\tau<4&\ \Longrightarrow\
			\frac{2m+1}{6}<\frac{9+\tau}{30},\\
			\frac{10-7m+4\tau}{27}
			&\le\frac{10-3\tau-7s}{27}
			\le\frac{10-3\beta-7s}{27}=\beta-\frac1{36180}.
		\end{aligned}
		\]
		Moreover,
		\[
		\frac{M}{N^{3/4}H^{1/4}}W^{1/3}
		=\left(\frac HM\right)^{1/2}
		\left(\frac{M}{\sqrt{HN}}W^{2/9}\right)^{3/2}
		\lesssim T^{(3\beta-\tau)/2}\le T^\beta.
		\]
		Equation~\eqref{eq:second} proves the proposition.
	\end{proof}
	
	\subsection{Verification of the second-spacing hypotheses}\label{sec:farey-inputs}
	The definition \eqref{eq:NR} gives
	\[
	2(R-1)^2N(2T)<C_3(2M)^3\le2R^2N(2T),
	\]
	the condition \cite[(2.6)]{Watt04}. Also $(2T)/(2M)^2\ge1/2$.
	
	The powers of $T$ in the three candidate lengths are
	\[
	N_{\mathrm I}\eqsim T^{m+(22\tau-17)/35},\qquad
	N_{\mathrm {II}}=T^{(m+5\tau-1)/7},\qquad
	N_{\mathrm {III}}=2T^{(5m+\tau-2)/3}.
	\]
	Each exponent is at least $443/4020>1/10$, so taking ceilings changes
	only fixed factors. In the first case, $m\ge(4+\tau)/10$. In the
	second, $m<(4+\tau)/10$ and $N\lesssim N_{\mathrm {II}},N_{\mathrm {III}}$.
	Consequently
	\[
	\frac{HNT}{M^3}
	\lesssim T^{(8-33\tau)/70}+T^{(1-3\tau-s)/3}
	\le 2T^{-37/4020}.
	\]
	Thus $M^3/(NT)\gtrsim HT^{37/4020}\to\infty$ and
	$R\eqsim(M^3/(NT))^{1/2}$. The remaining comparisons are
	\begin{equation}\label{eq:arc-hyp}
		\begin{aligned}
			\frac HN
			&\lesssim T^{(17-57\tau)/35}+T^{(2-6\tau-2s)/3}
			\lesssim T^{-37/2010},\\
			\frac RH
			&\lesssim T^{3(8\tau-3)/35}+T^{(3\tau-2s-1)/6}
			\lesssim T^{-27/2345},\\
			\frac NM
			&\lesssim T^{(22\tau-17)/35}+T^{(-\tau-6s-1)/7}
			\lesssim T^{-611/2345},\\
			\frac H{R^2}&\lesssim T^{-37/4020},\\
			\frac{HN^3T}{M^4}
			&\lesssim T^{(31\tau-16)/35}+T^{(\tau-1)/5}
			\lesssim T^{-43/335}.
		\end{aligned}
	\end{equation}
	For the last hypothesis of Watt's lemma, use
	\[
	\begin{gathered}
		\min(N_{\mathrm {II}},N_{\mathrm {III}})
		\le N_{\mathrm {II}}^{7/16}N_{\mathrm {III}}^{9/16}
		\eqsim T^{m+\tau/2-7/16},\\
		\frac{H^2N^3}{MR^4}\eqsim\frac{H^2N^5T^2}{M^7}
		\lesssim T^{(8\tau-3)/7}+T^{(8\tau-3)/16}
		\lesssim T^{-9/1072}.
	\end{gathered}
	\]
	Under the substitution $(M_*,H_*,T_*)=(2M,2H,2T)$, these
	comparisons verify Watt's conditions (2.4) and (2.7)--(2.9) for all
	sufficiently large $T$. Indeed, $m\ge\tau+s>0$ gives $M_*\ge4$,
	$H/R^2\to0$ and $R/H\to0$ give
	$4C_3\sqrt{H_*}\le R\le\varepsilon_0H_*$, while
	$H/N\to0$, $N/M\to0$, and $HN^3T/M^4\to0$ give (2.8)--(2.9)
	after increasing the fixed lower threshold for $T$. The last displayed
	estimate similarly gives (8.14), including its fixed constant. Thus,
	together with the exact verification of (2.6) above, all hypotheses of
	\cite[\Extref{Watt04}{Lemmas~8.1--8.2}]{Watt04} used below are satisfied. We also have
	\[
	N,R,W\to\infty,\qquad
	\frac{HN^2}{MR^2}\eqsim\frac{HN^3T}{M^4}=o(1).
	\]
	
	Extend the original $n$ interval by $2N$ to the left, and denote
	its consecutive length-$N$ intervals by $I_0,I_1,\ldots$. Apply
	\cite[\Extref{Watt04}{Section~2}, (2.10)--(2.23)]{Watt04} with modulus one to the
	derivative intervals $f'(I_i)$. Group the $I_i$ whose images form
	a long major arc, as in \cite[(2.13)]{Watt04}. For each remaining
	$I=I_i$ meeting the original summation range, one has $i\ge2$.
	Let $a_I\in\Z$ and $d_I\in\N$ be coprime, with $d_I\ge R$,
	so that $a_I/d_I$ is the selected fraction in $f'(I_{i-2})$. Let $m_I$ be an integer
	nearest the unique solution of $f'(x)=a_I/d_I$ in $I_{i-2}$.
	For $Q=2^jR$ with an integer $j\ge0$, let $\mathcal I_Q$ be the
	collection of these remaining intervals $I$ satisfying $Q\le d_I<2Q$. By \cite[(2.16), (2.23)]{Watt04} and
	\cite[(7.9)]{BW18},
	\begin{equation}\label{eq:Farey}
		n\in I\ \Longrightarrow\ N\le n-m_I\le3N,\qquad
		R\le d_I\lesssim H,\qquad
		\#\mathcal I_Q\lesssim\frac{MR^2}{NQ^2}.
	\end{equation}
	For each $I$, define
	\begin{equation}\label{eq:arc-data}
		\gamma_I=\tfrac12f''(m_I)\eqsim(NR^2)^{-1},\qquad
		b_I=\lfloor d_If(m_I)\rfloor,\qquad \kappa_I=d_If(m_I)-b_I.
	\end{equation}
	For each $I$, let $\bar a_I$ be the modular inverse specified by
	$0\le\bar a_I<d_I$ and $a_I\bar a_I\equiv1\pmod{d_I}$.
	Let $B_Q$ count the ordered pairs $(I,I')\in\mathcal I_Q^2$ satisfying
	\begin{equation}\label{eq:pairs}
		\begin{aligned}
			\left\|\frac{\bar a_I}{d_I}-\frac{\bar a_{I'}}{d_{I'}}\right\|
			&\le\frac{R^4}{HNQ^2},&
			\left|\frac{\gamma_{I'}d_{I'}^3}{\gamma_Id_I^3}-1\right|
			&\le\frac{R^2}{HN},\\
			\left\|\frac{\bar a_Ib_I}{d_I}-\frac{\bar a_{I'}b_{I'}}{d_{I'}}\right\|
			&\le\frac{R^2}{HQ},&
			|\kappa_I-\kappa_{I'}|&\le\frac QH.
		\end{aligned}
	\end{equation}
	For a sufficiently small
	fixed $c_*>0$, Watt's estimate \cite[\Extref{Watt04}{Lemma~8.2}, (8.17)--(8.20)]{Watt04}
	is
	\begin{equation}\label{eq:classical-spacing}
		B_Q\lesssim_\varepsilon T^\varepsilon
		\frac{M^2R^4}{H^2N^4}\left(\frac QH\right)^{2/3},
		\qquad R\le Q\le c_*H.
	\end{equation}
	Indeed, in the first case its additional factor is
	$1+(T^4/(64HM^9))^{1/5}=O(1)$. In the second, (8.20) applies.
	The ceilings give $N/2<N_{\mathrm I}\le N$ or
	$N/2<\min(N_{\mathrm {II}},N_{\mathrm {III}})\le N$ for large $T$.
	The conditions \eqref{eq:pairs} are \cite[(8.2), (8.5)--(8.7)]{Watt04}.
	Fixed changes in their tolerances follow from
	\eqref{eq:pair-comparison}, using $\log(\gamma_Id_I^3)$ in the second
	coordinate. The explanation on pp.~380--381 of \cite{Watt04} shows
	that this count uses only the function denoted there by $f(x,0)$.
	With parameters $(2M,2H,N,2T)$ and auxiliary function $\Phi$, that function is
	exactly the $f$ defined above.
	We invoke this arc-pair count, not a bound for Watt's original
	difference-phase exponential sum.
	
	\subsection{The double large sieve and the sum over denominators}
	The following definitions are used throughout the rest of this section.
	Fix a sufficiently large comparison constant $C_6\ge2$.
	Then choose a sufficiently large fixed $C_\eta>0$ and, if necessary,
	decrease the constant $c_*$ in \eqref{eq:classical-spacing}.
	For each denominator scale $Q=2^jR$, $j\ge0$, use all dyadic pairs
	$(K,L)$ satisfying the first line below, and define the other quantities by
	\begin{equation}\label{eq:spacing-scales}
		\begin{gathered}
			C_6^{-1}NQ/R^2\le K\le C_6NQ/R^2,\qquad
			C_6^{-1}HQ/R^2\le L\le C_6HQ/R^2,\\
			\eta=C_\eta R^2/(NH),\qquad Z=\eta K^2,\\
			\mathcal E_4=[KL^2(L+Z)]^{1/4},\qquad \mathcal E_6=(KL)^{1/2}.
		\end{gathered}
	\end{equation}
	There are $O(1)$ such rectangles $[K,2K)\times[L,2L)$.
	\Secref{sec:class-transform} shows that they cover all the
	transformed integer frequencies.
	For $I\in\mathcal I_Q$ and integers $k\approx K$, $l\approx L$, define
	\begin{equation}\label{eq:vectors}
		\begin{gathered}
			t_I=(\gamma_Id_I^3)^{-1/2},\qquad
			x_I=(\bar a_I/d_I,\{\bar a_Ib_I/d_I\},t_I,-\kappa_It_I/2),\\
			y_{kl}=(kl,l,l\sqrt k,l/\sqrt k).
		\end{gathered}
	\end{equation}
	The first two coordinates of $x_I$ lie in $[0,1)$. Taking the
	fractional part does not change $e(x_I\cdot y_{kl})$, since $kl,l\in\Z$.
	For each rectangle, \Secref{sec:class-transform} supplies finite
	families of coefficients $a_{I,j}$ and functions $b_j(k,l)$. They may
	depend on $Q,K,L$ and satisfy the coefficient bound below. If $\mathcal I_Q=\varnothing$, set $S_Q=0$. Otherwise define
	\begin{equation}\label{eq:separated}
		\begin{gathered}
			S_Q=\max_{(K,L)}\frac{R^2}{Q}\sum_j
			\left|\sum_{I\in\mathcal I_Q}\sum_{k\approx K,\,l\approx L}
			a_{I,j}b_j(k,l)e(x_I\cdot y_{kl})\right|,\\
			\sum_j\sup_I|a_{I,j}|\sup_{k,l}|b_j(k,l)|\lesssim\log^2T.
		\end{gathered}
	\end{equation}
	The maximum is over the dyadic pairs in \eqref{eq:spacing-scales}.
	Their number is bounded,
	so \Secref{sec:class-transform} gives
	\begin{equation}\label{eq:class-reduction}
		|S|\lesssim_\varepsilon T^\varepsilon M\sqrt{H/N}
		+\sum_{\substack{Q=2^jR\\j\ge0,\ R\le Q\le c_*H}}S_Q.
	\end{equation}
	\begin{proof}[Proof of \Lemref{lem:farey-Q}]
		Fix $Q=2^jR$ with $j\ge0$ and $R\le Q\le c_*H$.
		The case $S_Q=0$ is immediate. Otherwise fix a rectangle attaining
		the maximum. For a sufficiently large fixed
		constant $C$, choose the box widths in \eqref{eq:distributional-sieve} as
		\[
		\begin{gathered}
			D=(NR^2/Q^3)^{1/2},\qquad
			(D_1,D_2,D_3,D_4)=C(1,1,D,D),\\
			(E_1,E_2,E_3,E_4)=C(KL,L,L\sqrt K,L/\sqrt K).
		\end{gathered}
		\]
		Write $x=(x_1,\ldots,x_4)\in\mathcal X=\prod_{i=1}^4[-D_i,D_i]$
		for the integration variable. Then $t_I\eqsim D$, and the four products have sizes
		\[
		(D_iE_i)_{i=1}^4\eqsim
		\left(\frac{NHQ^2}{R^4},\frac{HQ}{R^2},
		\frac{NH}{R^2},\frac HQ\right).
		\]
		They are bounded below by fixed positive constants, so
		\[
		P:=\prod_{i=1}^4(1+D_iE_i)\eqsim\frac{H^4N^2Q^2}{R^8}.
		\]
		The first two sieve restrictions give the corresponding conditions in
		\eqref{eq:pairs}. The other two imply
		\[
		\begin{aligned}
			\left|\frac{\gamma_{I'}d_{I'}^3}{\gamma_Id_I^3}-1\right|
			&\lesssim\frac{|t_I-t_{I'}|}{D}
			\lesssim\frac1{DL\sqrt K}\eqsim\frac{R^2}{HN},\\
			|\kappa_I-\kappa_{I'}|
			&\lesssim\frac{|\kappa_It_I-\kappa_{I'}t_{I'}|}{D}
			+\frac{|t_I-t_{I'}|}{D}\\
			&\lesssim\frac{\sqrt K}{LD}+\frac1{DL\sqrt K}
			\lesssim\frac QH.
		\end{aligned}
		\]
		Thus the sieve pair count is $O(B_Q)$, by \eqref{eq:pair-comparison}.
		
		The cone hypotheses follow from
		\[
		L\gtrsim W\to\infty,\qquad K/L\eqsim N/H\to\infty,\qquad
		\eta K\eqsim C_\eta Q/H,\qquad
		\eta KL\eqsim C_\eta(Q/R)^2.
		\]
		The preceding choices of $C_\eta$ and $c_*$ ensure
		$1\le L<K\le\eta^{-1}\le KL$. Also
		$D\eqsim(\eta L\sqrt K)^{-1}$. After interchanging the first two
		coordinates, cover their box and the third-coordinate interval by a
		bounded number of translates of a box with side lengths
		$2,2,2(\eta L\sqrt K)^{-1}$. Translation multiplies the coefficients
		by constants of modulus one. At each fixed fourth coordinate $x_4$,
		absorb $e(x_4l/\sqrt k)$ into those coefficients.
		\Thmref{thm:first}, followed by integration in $x_4$, gives
		the distribution hypothesis of \Lemref{lem:elementary-sieve}
		with constants $c_q\lesssim_\varepsilon T^\varepsilon\mathcal E_q^q$
		for $q\in\{4,6\}$.
		Hence \eqref{eq:Farey} and \eqref{eq:classical-spacing} give
		\begin{equation}\label{eq:all-Q-mixed}
			\begin{aligned}
				S_Q&\lesssim_\varepsilon T^\varepsilon\frac{R^2}{Q}
				\sum_{q\in\{4,6\}}P^{1/q}(\#\mathcal I_Q)^{1-2/q}B_Q^{1/q}\mathcal E_q\\
				&\lesssim_\varepsilon T^\varepsilon\frac{MR}{N}
				\sum_{q\in\{4,6\}}(R/Q)^{3-20/(3q)}W^{4/(3q)}\mathcal E_q.
			\end{aligned}
		\end{equation}
		More explicitly,
		\[
		\begin{gathered}
			\mathcal E_6\eqsim\frac{\sqrt{NH}}R\frac QR,\\
			\mathcal E_4\lesssim
			\frac{N^{1/4}H^{3/4}}R\frac QR
			+\frac{N^{1/2}H^{1/4}}{R^{3/4}}\left(\frac QR\right)^{5/4},
		\end{gathered}
		\]
		and therefore
		\[
		\begin{aligned}
			\frac{S_Q}{H}\lesssim_\varepsilon T^\varepsilon\bigg[
			&\frac{M}{\sqrt{HN}}W^{2/9}(R/Q)^{8/9}\\
			&+\frac{M}{N^{3/4}H^{1/4}}W^{1/3}(R/Q)^{1/3}\\
			&+\frac{MR^{1/4}}{N^{1/2}H^{3/4}}W^{1/3}(R/Q)^{1/12}\bigg].
		\end{aligned}
		\]
		For each $\delta>0$, $\sum_{j\ge0}2^{-j\delta}=(1-2^{-\delta})^{-1}$.
		Sum the last display over $Q=2^jR$ and use
		\[
		\frac{MR^{1/4}}{N^{1/2}H^{3/4}}W^{1/3}
		=W^{-5/36}\frac{M}{\sqrt{HN}}W^{2/9},\qquad W\ge1.
		\]
		The error in \eqref{eq:class-reduction}, divided by $H$, is also
		bounded by the first term of \eqref{eq:second}.
	\end{proof}
	
	\subsection{The two van der Corput transformations}\label{sec:class-transform}
	We prove \eqref{eq:separated} and \eqref{eq:class-reduction}.
	By \eqref{eq:abel-bv}, it suffices to take $g=G=1$ and retain
	arbitrary subintervals.
	
	Let $\mathcal J$ correspond to a long major arc with reduced fraction
	$a_0/d_0$. By \cite[(2.10), (2.13)]{Watt04},
	\[
	d_0\lesssim R^2/H,\qquad
	|f'(x)-a_0/d_0|\lesssim(Hd_0)^{-1},\qquad
	f''(x)\eqsim(NR^2)^{-1}\quad(x\in\mathcal J).
	\]
	Thus the phase $g_h(x)=h(f(x)-a_0x/d_0)$ satisfies
	\[
	|g_h'(x)|\lesssim d_0^{-1},\qquad
	g_h''(x)\eqsim H/(NR^2)\quad(h\approx H,\ x\in\mathcal J).
	\]
	These are the derivative bounds used in the major-arc argument of
	\cite[(3.7)--(3.12), \Extref{Watt04}{Lemma~3.1}]{Watt04}, with auxiliary modulus one.
	Applied directly to $hf(x)$, that argument gives
	\[
	\begin{aligned}
		\left|\sum_{h\approx H}\sum_{n\in\mathcal J\cap\Z}e(hf(n))\right|
		&\lesssim \frac R{d_0}\sqrt{HN}
		+R\sqrt{N/H}+(H+d_0)\log(2d_0)\\
		&\lesssim\frac R{d_0}\sqrt{HN}\log T,
	\end{aligned}
	\]
	uniformly on subintervals. Indeed, $d_0\lesssim R^2/H$ and $R<H<N$ imply
	\[
	\frac{R\sqrt{N/H}}{(R/d_0)\sqrt{HN}}=\frac{d_0}H\lesssim1,\qquad
	\frac{H+d_0}{(R/d_0)\sqrt{HN}}\lesssim\frac R{\sqrt{HN}}\lesssim1.
	\]
	There are $O(d_0T/M^2)$ possible numerators and $O(1)$ intervals per
	fraction. Thus the total is
	\[
	\lesssim\log T\sum_{d_0\lesssim R^2/H}
	\frac{d_0T}{M^2}\frac R{d_0}\sqrt{HN}
	\lesssim\frac{MR}{\sqrt{HN}}\log T
	\le M\sqrt{H/N}\log T.
	\]
	
	Fix $I\in\mathcal I_Q$. While $I$ is fixed, write
	$a=a_I$, $d=d_I$, $\bar a=\bar a_I$, and retain the scales $K,L$
	from \eqref{eq:spacing-scales}, and put $x=n-m_I$.
	For $N\le x\le3N$, define
	\[
	E_I(x)=\left(f'(m_I)-\frac ad\right)x
	+\frac12\int_0^x(x-t)^2f'''(m_I+t)\,dt.
	\]
	Since $m_I$ is within $1/2$ of a solution of $f'=a/d$,
	\begin{equation}\label{eq:Taylor-weight}
		\begin{gathered}
			f(m_I+x)=\frac{b_I+\kappa_I}{d}+\frac adx+\gamma_Ix^2+E_I(x),\\
			|f'(m_I)-a/d|\lesssim(NR^2)^{-1},\qquad
			|f'''(m_I+x)|\lesssim(MNR^2)^{-1},\\
			H\bigl(\|E_I\|_{\infty,[N,3N]}+N\|E_I'\|_{\infty,[N,3N]}\bigr)
			\lesssim H/R^2+HN^2/(MR^2)=o(1).
		\end{gathered}
	\end{equation}
	Consequently
	$H^iN^j|\partial_h^i\partial_x^je(hE_I(x))|\lesssim1$
	for $H\le h\le2H$, $N\le x\le3N$ and $i,j\in\{0,1\}$.
	For real endpoints satisfying
	$H\le H_-\le H_+\le2H$ and $N\le N_-\le N_+\le3N$,
	\eqref{eq:abel-two} reduces the sum to
	\[
	\sum_{H_-\le h\le H_+}\sum_{N_-\le x\le N_+}
	e\left(\frac{ahx+(b_I+\kappa_I)h}{d}+\gamma_Ihx^2\right).
	\]
	
	For $l\in\N$, choose $h_l\in\{0,\ldots,d-1\}$ with
	$h_l\equiv-\bar al\pmod d$. For $h>0$, define
	\[
	\begin{gathered}
		J_l=[H_-,H_+]\cap[l/(2\gamma_IdN_+),l/(2\gamma_IdN_-)],\\
		\phi_{I,l}(h)=\kappa_Ih/d-l^2/(4\gamma_Id^2h),\qquad
		A_I(h)=(2\gamma_Ih)^{-1/2}.
	\end{gathered}
	\]
	The first transformation \cite[\Extref{Watt04}{Lemma~4.1}, (4.22)]{Watt04}, with
	its auxiliary modulus equal to one and its character identically one, gives
	\[
	\sum_{l\in\N}e(1/8-\bar ab_Il/d)
	\sum_{h\in J_l\cap(h_l+d\Z)}A_I(h)e(\phi_{I,l}(h))
	+O\bigl((Q/R)\sqrt{HN}\log(2HN)\bigr).
	\]
	Only $l\eqsim L$ occur. Write $h=h_l+dv$. Summation uses
	$v\in\Z$. For real $v$ with $H\le h\le2H$ and $j=2,3,4$,
	\[
	\begin{gathered}
		\partial_v^j\phi_{I,l}(h_l+dv)
		=\frac{(-1)^{j+1}j!l^2d^{j-2}}{4\gamma_Ih^{j+1}},\\
		|\partial_v^j\phi_{I,l}(h_l+dv)|
		\eqsim\frac{NH}{R^2}(d/H)^j,\qquad
		\|A_I\|_{\infty,[H,2H]}+\Var_{[H,2H]}(A_I)\lesssim R\sqrt{N/H}.
	\end{gathered}
	\]
	Apply \cite[\Extref{Huxley96}{Lemma~5.5.3}]{Huxley96} in $v$, at scale $H/d$.
	The modulus of its second derivative is comparable to
	$Nd^2/(HR^2)\gtrsim N/H\to\infty$.
	For $k,l\in\N$, define $h_*,x_*>0$ by the formulas below.
	They are well-defined because $0\le\kappa_I<1$, and they satisfy
	\[
	\begin{gathered}
		d\phi_{I,l}'(h_*)=k,\qquad
		h_* =\frac l{2\sqrt{\gamma_Id(k-\kappa_I)}},\qquad
		x_* =\sqrt{\frac{k-\kappa_I}{\gamma_Id}}.\\[-2pt]
		\phi_{I,l}(h_*)-kh_*/d
		=-\frac{l\sqrt{k-\kappa_I}}{\sqrt{\gamma_Id^3}},\qquad
		\frac{A_I(h_*)}{d\sqrt{|\phi_{I,l}''(h_*)|}}
		=\frac1{2\sqrt{\gamma_Id(k-\kappa_I)}}.
	\end{gathered}
	\]
	The restrictions $x_*\in[N_-,N_+]$ imply $k\eqsim K$.
	The progression contributes
	$e(kh_l/d)=e(-\bar akl/d)$. The factors $e(1/8)$ and $e(-1/8)$ cancel.
	The transformed sum is
	\begin{equation}\label{eq:two-Poisson}
		\sum_{\substack{k,l\in\N\\x_*\in[N_-,N_+],\ h_*\in[H_-,H_+]}}
		\frac{e(-\bar al(k+b_I)/d-l\sqrt{k-\kappa_I}/\sqrt{\gamma_Id^3})}
		{2\sqrt{\gamma_Id(k-\kappa_I)}}
		+O\left(\frac QR\sqrt{HN}\log(2HN)\right).
	\end{equation}
	The second transformation costs
	$O(LR\sqrt{N/H}\log(2HN))$. Endpoint changes cost
	$O(LR\sqrt{N/H}+LR^2/Q)$. Both are included in \eqref{eq:two-Poisson}.
	
	When $d\ge c_*H$, the first transformation has
	$O(L(1+H/d))$ terms, each of size $O(R\sqrt{N/H})$.
	Using \eqref{eq:Farey}, the large denominators and all transformation
	errors contribute at most
	\[
	\begin{aligned}
		&\sum_{\substack{Q=2^jR\\j\ge0,\ R\le Q\lesssim H}}
		\frac{MR^2}{NQ^2}\frac QR\sqrt{HN}\log(2HN)
		+\frac{MR}{\sqrt{HN}}\\
		&\hspace{25mm}\lesssim M\sqrt{H/N}\log(2HN).
	\end{aligned}
	\]
	For the remaining denominators, replace $k-\kappa_I$ by $k$ only in
	the interval restrictions of \eqref{eq:two-Poisson}. For fixed $l$,
	the restrictions specify one interval in $k-\kappa_I$. Translation by
	$0\le\kappa_I<1$ changes at most two integers. The cost is
	\[
	O(LR^2/Q)=O(H)\quad\hbox{per interval},\qquad
	O(MH/N)\le O(M\sqrt{H/N})\quad\hbox{in total}.
	\]
	For real $k>\kappa_I$ and $l>0$, define
	\[
	\mathcal R_I(k,l)=t_Il\left(\sqrt{k-\kappa_I}-\sqrt k
	+\frac{\kappa_I}{2\sqrt k}\right).
	\]
	After conjugation, the exponential factor is
	$e(x_I\cdot y_{kl}+\mathcal R_I(k,l))$.
	For real $K\le k\le2K$, $L\le l\le2L$ and $i,j\in\{0,1\}$,
	\[
	K^iL^j|\partial_k^i\partial_l^j\mathcal R_I|
	\lesssim t_ILK^{-3/2}
	\eqsim\frac{HR^2}{NQ^2}\lesssim H/N=o(1).
	\]
	On the same rectangle, the normalized amplitude satisfies
	\[
	K^iL^j\left|\partial_k^i\partial_l^j
	\left(\frac{Qe(\mathcal R_I(k,l))}
	{2R^2\sqrt{\gamma_Id(k-\kappa_I)}}\right)\right|\lesssim1
	\qquad(i,j\in\{0,1\}).
	\]
	These are the bounds required by \eqref{eq:abel-two}.
	The remaining restrictions are
	\[
	\sqrt k\in[\sqrt{\gamma_Id}\,N_-,\sqrt{\gamma_Id}\,N_+],\qquad
	l/\sqrt k\in[2\sqrt{\gamma_Id}\,H_-,2\sqrt{\gamma_Id}\,H_+].
	\]
	Apply \Lemref{lem:finite-separation} to these two restrictions and
	\eqref{eq:abel-two} to the amplitude. Their total coefficient bound is
	$O(\log^2(2+KL))=O(\log^2T)$, proving \eqref{eq:separated}.
	To sum the absolute values over $I$, first multiply each interval sum
	by a scalar of modulus at most one making it nonnegative.
	The preceding error bounds give \eqref{eq:class-reduction}.
	
	\section{A distribution estimate for exponential sums on the cone}\label{sec:spacing}
	We combine Guth--Maldague's
	amplitude-dependent estimate \cite[\Extref{GM}{Theorem~2}]{GM} with an integer count.
	The purpose of this section is to convert that wave-envelope estimate into
	the two-term distribution bound in \eqref{eq:mixed}. The geometric input and
	the subsequent lattice count are kept separate so that the dependence of the
	fourth- and sixth-power terms can be read directly.
	
	\subsection{Reduction to two estimates}\label{sec:firstspacing}
	Let $K,L,\eta>0$ satisfy
	\begin{equation}\label{eq:first-hyp}
		1\le L<K\le\eta^{-1}\le KL,\qquad Z=\eta K^2.
	\end{equation}
	For complex coefficients $c=(c_{kl})$ with $|c_{kl}|\le1$, indexed by integers
	$k\approx K$, $l\approx L$, and for $x\in\R^3$, $\lambda>0$, define
	\[
	\begin{gathered}
		\mathcal B=[-1,1]^2\times
		[-(\eta L\sqrt K)^{-1},(\eta L\sqrt K)^{-1}],\\
		F_c(x)=\sum_{k\approx K,\,l\approx L}c_{kl}
		e(lx_1+klx_2+l\sqrt k\,x_3),\\
		\mathcal D(\lambda)=\sup_{|c_{kl}|\le1}
		\frac{|\{x\in\mathcal B:|F_c(x)|>\lambda\}|}{|\mathcal B|}
		\qquad(\lambda>0).
	\end{gathered}
	\]
	Zero coefficients allow arbitrary interval restrictions.
	\begin{theorem}[First spacing]\label{thm:first}
		Under \eqref{eq:first-hyp}, for every $\lambda>0$,
		\begin{equation}\label{eq:mixed}
			\mathcal D(\lambda)\lesssim_\varepsilon(KL)^\varepsilon
			\left((KL)^3\lambda^{-6}+KL^2(L+Z)\lambda^{-4}\right).
		\end{equation}
	\end{theorem}
	
	If $K\le4$, then $L<K\le4$ and $|F_c|\lesssim KL$ give
	\eqref{eq:mixed}. Henceforth assume $K>4$. Fix coefficients
	$c=(c_{kl})$ as above, a level $\lambda>0$ and a sufficiently small
	constant $c_0>0$. Choose
	$\varphi\in\mathscr S(\R)$ with
	$\operatorname{supp}\widehat\varphi\subseteq[-c_0,c_0]$ and
	$|\varphi|\ge1$ on $[-2,2]$. Put
	\[
	\chi(x)=\varphi(x_1/K)\varphi(x_2)\varphi(\eta L\sqrt K\,x_3),\qquad
	f=\chi F_c,\qquad V=\frac{\sqrt K}{\eta L}=\frac{K|\mathcal B|}{8}.
	\]
	Periodicity in $x_1$ and $|\chi|\ge1$ on
	$[-K,K]\times[-1,1]\times[-(\eta L\sqrt K)^{-1},(\eta L\sqrt K)^{-1}]$
	give
	\begin{equation}\label{eq:cone-normalization}
		V\frac{|\{x\in\mathcal B:|F_c(x)|>\lambda\}|}{|\mathcal B|}
		\lesssim |\{x\in\R^3:|f(x)|>\lambda\}|.
	\end{equation}
	For $\xi=(\xi_1,\xi_2,\xi_3)\in\R^3$, use normalized frequency coordinates
	\[
	\zeta=\left(\frac{\xi_1}{L},\frac{\xi_2}{KL},
	\frac{\xi_3}{L\sqrt K}\right).
	\]
	In these coordinates, the frequencies of $F_c$ satisfy $\zeta_1\zeta_2=\zeta_3^2$.
	The linear map
	\[
	\zeta\longmapsto
	\left(\frac{\zeta_3}{2},\frac{\zeta_2-\zeta_1}{4},
	\frac{\zeta_2+\zeta_1}{4}\right)
	\]
	maps them into the truncated cone
	\[
	\{(z_1,z_2,z_3)\in\R^3:z_1^2+z_2^2=z_3^2,\ 1/2\le z_3\le3/2\}.
	\]
	The standard cone in \cite[\Extref{GM}{Theorem~2}]{GM} is truncated to height
	between $1/2$ and $1$. A fixed finite radial subdivision, followed
	by fixed dilations, reduces the displayed height interval to this
	model. We apply the estimate to the resulting pieces separately.
	The associated smooth frequency cutoffs preserve the Fourier $L^1$
	and spatial $L^2$ bounds on the projected summands used below, up to
	absolute factors. Thus the subdivisions do not change any power of
	$K,L$ or $\eta$.
	In the $\zeta$ coordinates, $\operatorname{supp}\widehat\chi$ lies in
	$[-c_0/(KL),c_0/(KL)]^2\times[-c_0\eta,c_0\eta]$.
	Since $(KL)^{-1}\le\eta$, choose a power of four $\mathcal R\ge4$ with
	$\mathcal R\le\eta^{-1}<4\mathcal R$ and apply
	\cite[\Extref{GM}{Theorem~2}, \Extref{GM}{Section~5}, (29)--(30)]{GM}.
	
	Use the sectors and boxes in \cite[\Extref{GM}{Section~5}, items (1)--(3)]{GM},
	transported to the original coordinates. A finest sector $\theta$
	has angular width comparable to $\mathcal R^{-1/2}$. Define
	$\widehat{f_\theta}=\mathbf1_\theta\widehat f$. All sums in $\theta$
	are over this finite partition, with the displayed containment restrictions.
	For $\sigma\in\{1,1/2,\ldots,\mathcal R^{-1/2}\}$, put
	$\Delta=K\sigma$. Let $\mathcal P_\Delta$ consist of the sectors
	$\vartheta$ of angular width comparable to $\sigma$ that meet
	$\operatorname{supp}\widehat f$. Then
	\[
	\sqrt Z\lesssim\Delta\lesssim K,\qquad
	\#\mathcal P_\Delta\lesssim K/\Delta,
	\]
	and the $k$ indices contributing to each $\vartheta$ lie in an interval
	of length $O(\Delta)$. The scales $\Delta$ have ratio two.
	
	For each scale $\Delta$ and $\vartheta\in\mathcal P_\Delta$,
	let $B_\vartheta$ be the centered parallelepiped obtained by
	transporting the box in \cite[\Extref{GM}{Section~5}, item~(2)]{GM} to the $x$
	coordinates. Write
	$B_\vartheta=A_\vartheta[-1,1]^3$ with $A_\vartheta$ invertible,
	and define the dual box $B_\vartheta^*=A_\vartheta^{-\mathsf T}[-1,1]^3$.
	The translates $\mathcal Q=z+B_\vartheta$, with
	$z\in2A_\vartheta\Z^3$, tile $\R^3$ up to boundaries.
	Let $W_{\mathcal Q}$ be the corresponding weight from
	\cite[\Extref{GM}{Section~5}, item~(3)]{GM}. In the adapted box coordinates,
	this weight is an affine pullback of the fixed polynomial product
	weight of \cite[(13)]{GM}. Its comparison constants are uniform in
	the translation, orientation and dimensions of $\mathcal Q$.
	We make the Fourier-supported majorant explicit, as in
	\cite[\Extref{GM}{Definition~4}, (36)]{GM}. For $u=(u_1,u_2,u_3)\in\R^3$, put
	\[
	w(u)=\prod_{j=1}^3(1+u_j^2)^{-100}.
	\]
	Choose a fixed nonnegative $\psi_0\in\mathscr S(\R^3)$ with
	$\int\psi_0=1$, $\psi_0>0$ on $[-1,1]^3$, and
	$\operatorname{supp}\widehat\psi_0\subseteq[-1,1]^3$.
	Set $\widetilde w=\psi_0*w$ and $c_w=\int\widetilde w>0$. Then
	\[
	\widetilde w\eqsim w,\qquad
	\operatorname{supp}\widehat{\widetilde w}\subseteq[-1,1]^3,\qquad
	w(u+v)\eqsim w(u)\quad(|v|_\infty\le1).
	\]
	The first comparison follows by integrating
	$w(u-v)\lesssim w(u)\prod_{j=1}^3(1+v_j^2)^{100}$ against $\psi_0(v)$
	and using the last comparison for the lower bound.
	Define
	\[
	\psi_\vartheta(x)=
	\frac{\widetilde w(A_\vartheta^{-1}x)}{c_w|\det A_\vartheta|},\qquad
	\widehat\psi_\vartheta(\xi)=
	c_w^{-1}\widehat{\widetilde w}(A_\vartheta^{\mathsf T}\xi).
	\]
	For $\mathcal Q=z+B_\vartheta$, the weights above satisfy
	$W_{\mathcal Q}(y)\lesssim w(A_\vartheta^{-1}(y-z))$.
	Since $|\mathcal Q|=8|\det A_\vartheta|$ and
	$|A_\vartheta^{-1}(x-z)|_\infty\le1$ for $x\in\mathcal Q$, we obtain
	\begin{equation}\label{eq:cone-kernel}
		\begin{gathered}
			\psi_\vartheta\ge0,\qquad \|\psi_\vartheta\|_1=1,\qquad
			\operatorname{supp}\widehat\psi_\vartheta\subseteq B_\vartheta^*,\\
			|\mathcal Q|^{-1}W_{\mathcal Q}(y)
			\lesssim |\det A_\vartheta|^{-1}w(A_\vartheta^{-1}(x-y))
			\lesssim\psi_\vartheta(x-y)\quad(x\in\mathcal Q).
		\end{gathered}
	\end{equation}
	Define
	\[
	T_\vartheta=\psi_\vartheta*
	\sum_{\theta\subseteq\vartheta}|f_\theta|^2\ge0.
	\]
	For every translate $\mathcal Q=z+B_\vartheta$ in this tiling,
	\[
	\frac1{|\mathcal Q|}\int_{\R^3}\sum_{\theta\subseteq\vartheta}|f_\theta(y)|^2W_{\mathcal Q}(y)\,dy
	\lesssim\inf_{x\in\mathcal Q}T_\vartheta(x).
	\]
	The selection condition in \cite[\Extref{GM}{Theorem~2}]{GM} and
	$\#\mathcal P_\Delta\lesssim K/\Delta$ therefore imply that, for each $\delta>0$, there is a fixed
	constant $c_\delta>0$ such that
	\begin{equation}\label{eq:GMcone}
		\lambda^4|\{|f|>\lambda\}|
		\lesssim_\delta(KL)^\delta
		\sum_\Delta\sum_{\vartheta\in\mathcal P_\Delta}
		\int_{\{T_\vartheta\ge c_\delta(KL)^{-\delta}
			\lambda^2\Delta^2/K^2\}}T_\vartheta^2.
	\end{equation}
	In the notation of \cite[\Extref{GM}{Theorem~2}]{GM}, $\#\tau$ counts the
	\emph{active coarse sectors at scale} $\sigma=\Delta/K$, not the
	finest sectors lying inside a particular coarse sector. The bound
	$\#\mathcal P_\Delta\lesssim K/\Delta$ gives
	$(\#\tau)^{-2}\gtrsim\Delta^2/K^2$. Thus every envelope selected by
	Theorem~2 meets the weaker threshold used on the right of
	\eqref{eq:GMcone}. The fixed linear normalization above takes the
	$O(\eta)$ Fourier neighborhood to the $O(\mathcal R^{-1})$ cone
	neighborhood required there.
	Indeed, the preceding average estimate gives
	\[
	\frac1{|\mathcal Q|}
	\left(\int_{\R^3}\sum_{\theta\subseteq\vartheta}
	|f_\theta(y)|^2W_{\mathcal Q}(y)\,dy\right)^2
	\lesssim\int_{\mathcal Q}T_\vartheta(x)^2\,dx.
	\]
	The selection condition places each retained $\mathcal Q$ inside
	the integration set in \eqref{eq:GMcone}. The translates tile space
	and $\mathcal R\le KL$.
	
	\begin{lemma}\label{lem:localization}
		For $\vartheta\in\mathcal P_\Delta$,
		\begin{equation}\label{eq:envelope-counts}
			V^{-1}\|T_\vartheta\|_2^2
			\lesssim_\varepsilon(KL)^\varepsilon\Delta L^2(\Delta+L+Z).
		\end{equation}
		Also,
		\begin{equation}\label{eq:coarse-energy}
			\Delta\ge L+Z\quad\Longrightarrow\quad
			\|T_\vartheta\|_\infty\lesssim\Delta L.
		\end{equation}
	\end{lemma}
	
	\begin{proof}[Proof of \Thmref{thm:first}]
		Apply \Lemref{lem:localization} with loss $\delta>0$.
		Since $\#\mathcal P_\Delta\lesssim K/\Delta$,
		\[
		V^{-1}\sum_{\vartheta\in\mathcal P_\Delta}\|T_\vartheta\|_2^2
		\lesssim_\delta(KL)^\delta KL^2(\Delta+L+Z).
		\]
		For $\Delta<L+Z$, summing over $O(\log(2K))$ scales gives
		\[
		V^{-1}\sum_{\Delta<L+Z}\sum_{\vartheta\in\mathcal P_\Delta}
		\|T_\vartheta\|_2^2
		\lesssim_\delta(KL)^{2\delta}KL^2(L+Z).
		\]
		For $\Delta\ge L+Z$, a nonempty integration set in \eqref{eq:GMcone}
		requires
		\[
		c_\delta(KL)^{-\delta}\lambda^2\Delta^2/K^2
		\le\|T_\vartheta\|_\infty\lesssim\Delta L,
		\qquad
		\Delta\lesssim_\delta(KL)^\delta K^2L/\lambda^2.
		\]
		The geometric sum of these $\Delta$ is bounded by twice its largest
		term. Substitution in \eqref{eq:GMcone}, followed by
		\eqref{eq:cone-normalization}, yields
		\[
		\lambda^4\mathcal D(\lambda)
		\lesssim_\delta(KL)^{3\delta}
		\left(KL^2(L+Z)+\frac{K^3L^3}{\lambda^2}\right).
		\]
		Take $\delta=\varepsilon/4$.
	\end{proof}
	
	\subsection{Fourier support and the integer count}
	\begin{proof}[Proof of \Lemref{lem:localization}]
		For each finest sector $\theta$ and integers $k\approx K$, $l\approx L$,
		define $u_{\theta,kl}$ by
		\[
		\widehat u_{\theta,kl}(\xi)
		=\mathbf1_\theta(\xi)\widehat\chi\bigl(\xi-(l,kl,l\sqrt k)\bigr),
		\qquad f_\theta=\sum_{k\approx K,\,l\approx L}c_{kl}u_{\theta,kl}.
		\]
		Fourier inversion and Plancherel give
		\[
		\|u_{\theta,kl}\|_\infty\le\|\widehat\chi\|_1\lesssim1,\qquad
		\|u_{\theta,kl}\|_2\le\|\chi\|_2\lesssim V^{1/2},\qquad
		\|u_{\theta,kl}\|_4^4\lesssim V.
		\]
		
		Fix $\vartheta\in\mathcal P_\Delta$. Choose an interval in
		$[K,2K]$ of length $O(\Delta)$ containing every $k$ for which
		$u_{\theta,kl}\not\equiv0$ for some $\theta\subseteq\vartheta$ and $l\approx L$.
		Let $k_0$ be its midpoint, and put $t_0=\sqrt{k_0/K}$.
		Use the frequency coordinates
		\[
		(\zeta_1,\ \zeta_3-t_0\zeta_1,\
		\zeta_2-2t_0\zeta_3+t_0^2\zeta_1).
		\]
		The frequency $(l,kl,l\sqrt k)$ has coordinates
		\[
		\left(\frac lL,\frac lL(\sqrt{k/K}-t_0),
		\frac lL(\sqrt{k/K}-t_0)^2\right),
		\]
		and the widths of $B_\vartheta^*$ are comparable to
		$\eta/\sigma^2,\eta/\sigma,\eta$, where $\sigma=\Delta/K$.
		For $\theta\subseteq\vartheta$ and integers $k,k'\approx K$, $l,l'\approx L$,
		the condition
		$\psi_\vartheta*(u_{\theta,kl}\overline{u_{\theta,k'l'}})\not\equiv0$
		and \eqref{eq:cone-kernel} imply
		\[
		\begin{gathered}
			|l'-l|\lesssim LZ/\Delta^2,\\
			|l'(\sqrt{k'}-\sqrt{k_0})-l(\sqrt k-\sqrt{k_0})|
			\lesssim LZ/(\Delta\sqrt K).
		\end{gathered}
		\]
		The widths of $\widehat\chi$ are absorbed because $\eta KL\ge1$ and
		$\sigma\le1$. As $l'\eqsim L$ and
		$|\sqrt k-\sqrt{k_0}|\lesssim\Delta/\sqrt K$,
		\[
		\frac{L|k'-k|}{\sqrt K}
		\lesssim\frac{LZ}{\Delta\sqrt K}
		+\frac{LZ}{\Delta^2}\frac\Delta{\sqrt K}
		\lesssim\frac{LZ}{\Delta\sqrt K}.
		\]
		Thus
		\begin{equation}\label{eq:unrounded-pairs}
			|k'-k|\lesssim Z/\Delta,\qquad |l'-l|\lesssim LZ/\Delta^2.
		\end{equation}
		
		Each translated support meets $O(1)$ finest sectors. For a fixed
		$(k,l)$, \eqref{eq:unrounded-pairs} allows at most
		$C(1+Z/\Delta)(1+LZ/\Delta^2)$ choices of $(k',l')$.
		Since there are $O(\Delta L)$ choices of $(k,l)$,
		\[
		\|T_\vartheta\|_\infty
		\lesssim\Delta L(1+Z/\Delta)(1+LZ/\Delta^2)
		\lesssim\Delta L\qquad(\Delta\ge L+Z).
		\]
		This proves \eqref{eq:coarse-energy}.
		
		For every pair of projected summands, Young's and H\"older's
		inequalities give
		\[
		\|\psi_\vartheta*(u_{\theta,kl}\overline{u_{\theta,k'l'}})\|_2
		\le\|u_{\theta,kl}\|_4\|u_{\theta,k'l'}\|_4\lesssim V^{1/2}.
		\]
		Each term in the expansion of $\int_{\R^3}T_\vartheta(x)^2\,dx$
		is therefore $O(V)$. In a nonzero term, the integers $k_i,k_i'$
		lie in the fixed interval for $\vartheta$, and $l_i,l_i'\approx L$
		for $i=1,2$. Fourier support gives
		\[
		\left|\sum_{i=1}^2(l_i'-l_i)\right|\le4c_0/K<1,\qquad
		\left|\sum_{i=1}^2(k_i'l_i'-k_il_i)\right|\le4c_0<1.
		\]
		Both sums vanish. Define
		$d_i=k_i'-k_i$ for $i=1,2$, and $v=l_1'-l_1$. Then
		\begin{equation}\label{eq:paired}
			l_2'-l_2=-v,\qquad
			l_1d_1+l_2d_2=v(k_2'-k_1'),\qquad
			|d_i|\lesssim Z/\Delta\quad(i=1,2).
		\end{equation}
		For this upper bound, discard the $l$-difference restriction in
		\eqref{eq:unrounded-pairs}.
		
		If $(d_1,d_2)\ne(0,0)$, interchange the two pairs so that $d_1\ne0$.
		Choose $k_1',k_2',v,d_2,l_2$ first. The number of choices is
		\[
		O(\Delta^2)\,O(L)\,O(1+Z/\Delta)\,O(L)
		=O\bigl(\Delta^2L^2(1+Z/\Delta)\bigr).
		\]
		For an admissible choice, the nonzero integer
		$D=v(k_2'-k_1')-l_2d_2$ satisfies
		\[
		l_1d_1=D,\qquad
		0<|D|=|l_1d_1|\lesssim L(1+Z/\Delta)\lesssim KL.
		\]
		The divisor bound \cite[\Extref{HW}{Theorem~315}]{HW} gives
		$O_\varepsilon((KL)^\varepsilon)$ choices of $(l_1,d_1)$.
		All remaining indices are determined by
		\[
		k_i=k_i'-d_i\quad(i=1,2),\qquad
		l_1'=l_1+v,\qquad l_2'=l_2-v.
		\]
		If $d_1=d_2=0$, \eqref{eq:paired} reduces to
		$v(k_2'-k_1')=0$. The cases $v=0$ and $v\ne0$ give, respectively,
		\[
		O(\Delta^2L^2),\qquad O(\Delta L^3).
		\]
		Consequently
		\[
		\begin{aligned}
			V^{-1}\|T_\vartheta\|_2^2
			&\lesssim_\varepsilon(KL)^\varepsilon
			\bigl(\Delta^2L^2(1+Z/\Delta)+\Delta L^3\bigr)\\
			&=(KL)^\varepsilon\Delta L^2(\Delta+L+Z),
		\end{aligned}
		\]
		which is \eqref{eq:envelope-counts}.
	\end{proof}
	
	\appendix
	
	\section{Summation and pair counts}\label{sec:transforms}
	\subsection{Abel summation}
	Let $u,v\in\Z$ with $u\le v$, and let $b_u,\ldots,b_v$ be complex.
	For $u\le n\le v$, discrete Abel summation gives
	\begin{equation}\label{eq:abel-bv}
		b_n=b_v+\sum_{j=u}^{v-1}(b_j-b_{j+1})\mathbf1_{n\le j},\qquad
		|b_v|+\sum_{j=u}^{v-1}|b_j-b_{j+1}|
		\le\|b\|_\infty+\Var(b).
	\end{equation}
	Thus estimates uniform on subintervals extend to bounded-variation
	weights, using their actual values at jumps.
	Let $\mathcal I$ be a nonempty index set, and let
	$\mathcal Q\subseteq\R^2$ be a closed rectangle with side lengths $X,Y\ge1$
	in the $x,y$ directions. Suppose the complex-valued functions
	$A_I$, $I\in\mathcal I$, are twice continuously differentiable on a
	neighborhood of $\mathcal Q$. With all norms below taken on $\mathcal Q$,
	assume
	\[
	\sup_I\big(\|A_I\|_\infty+X\|\partial_xA_I\|_\infty
	+Y\|\partial_yA_I\|_\infty
	+XY\|\partial_x\partial_yA_I\|_\infty\big)\le C.
	\]
	Applying \eqref{eq:abel-bv} in each coordinate gives finite families
	of integers $k_j,l_j$ and coefficients $a_{I,j}$ such that, on
	$\mathcal Q\cap\Z^2$,
	\begin{equation}\label{eq:abel-two}
		A_I(k,l)=\sum_j a_{I,j}\mathbf1_{k\le k_j}\mathbf1_{l\le l_j},
		\qquad \sum_j\sup_I|a_{I,j}|\lesssim C.
	\end{equation}
	The integers $k_j,l_j$ are independent of $I$.
	The coefficients are endpoint values and first or mixed differences,
	whose sums of suprema are bounded by the displayed derivative bounds.
	
	The van der Corput transformations are cited directly from
	\cite[\Extref{Huxley96}{Lemma~5.5.3}]{Huxley96}. Real and imaginary parts allow complex
	weights, and conjugation allows either sign of the second derivative.
	For a bounded-variation amplitude $A$ on an interval $I$ and a real
	phase $\phi$ with $|\phi''|\eqsim\Lambda>0$ on $I$, changing
	endpoint conventions costs
	$O((\|A\|_{\infty,I}+\Var_I(A))(1+\Lambda^{-1/2}))$, which is included
	in the cited error. The scale in that lemma is kept fixed on subintervals.
	
	\subsection{Interval restrictions}
	\begin{lemma}\label{lem:finite-separation}
		Let $\Omega$ and $\mathcal I$ be nonempty finite sets, let
		$c_1,c_2:\Omega\to\R$, and let $J_{I,1},J_{I,2}\subseteq\R$
		be intervals for $I\in\mathcal I$.
		There are finite families of coefficients $a_{I,j}$ and functions
		$b_j:\Omega\to\mathbb C$ such that, for every $I\in\mathcal I$
		and $\nu\in\Omega$,
		\begin{equation}\label{eq:finite-separation}
			\begin{gathered}
				\mathbf1_{J_{I,1}}(c_1(\nu))\mathbf1_{J_{I,2}}(c_2(\nu))
				=\sum_j a_{I,j}b_j(\nu),\\
				\sum_j\sup_I|a_{I,j}|\sup_\nu|b_j(\nu)|\lesssim\log^2(2+\#\Omega).
			\end{gathered}
		\end{equation}
		For one interval restriction the bound is $O(\log(2+\#\Omega))$.
	\end{lemma}
	\begin{proof}
		For $c:\Omega\to\R$, define
		\[
		N_c=\#c(\Omega),\qquad
		r_c(\nu)=\#\{t\in c(\Omega):t<c(\nu)\}.
		\]
		For an interval $J\subseteq\R$, define
		$A_J=\{r_c(\nu):\nu\in\Omega,\ c(\nu)\in J\}$. This is a possibly empty interval
		of integers in $\{0,\ldots,N_c-1\}$. For $0\le j<N_c$, set
		$\gamma_J(j)=N_c^{-1}\sum_{v\in A_J}e(-jv/N_c)$.
		Finite Fourier inversion and summation of a geometric progression give
		\[
		\begin{gathered}
			\mathbf1_J(c(\nu))=
			\sum_{j=0}^{N_c-1}\gamma_J(j)e(jr_c(\nu)/N_c),\\
			|\gamma_J(j)|\lesssim\frac1{1+\min(j,N_c-j)},\qquad
			\sum_{j=0}^{N_c-1}\sup_J|\gamma_J(j)|\lesssim\log(2N_c).
		\end{gathered}
		\]
		Apply this to $c_1,c_2$ and multiply. Repeated values have the same
		rank, and all endpoint conventions are retained.
	\end{proof}
	
	\subsection{Pair counts}
	Let $d,n,r\in\N$, $x_j=(x_{j,1},\ldots,x_{j,d})\in\R^d$ for $1\le j\le n$, and
	$\delta=(\delta_1,\ldots,\delta_d)\in(0,\infty)^d$. Define
	\[
	B(\delta)=\#\{(j,j')\in\{1,\ldots,n\}^2:
	|x_{j,i}-x_{j',i}|\le\delta_i\ (1\le i\le d)\}.
	\]
	Partition the indices into $r$ groups, and let $B_\ell(\delta)$
	count the pairs within group $\ell$. For every fixed $A\ge1$,
	\begin{equation}\label{eq:pair-comparison}
		B(A\delta)\lesssim_{A,d}r\sum_{\ell=1}^r B_\ell(\delta).
	\end{equation}
	For $\nu\in\Z^d$ and $1\le\ell\le r$, let $n_{\nu,\ell}$
	count group $\ell$ in the half-open box
	$\prod_{i=1}^d[\nu_i\delta_i,(\nu_i+1)\delta_i)$, and put
	$n_\nu=\sum_{\ell=1}^r n_{\nu,\ell}$. All sums over $\nu,\mu$
	below run through $\Z^d$, and $\ell$ runs from $1$ to $r$. Then
	\[
	B(A\delta)\le\sum_{|\nu-\mu|_\infty\le\lceil A\rceil+1}n_\nu n_\mu
	\lesssim_{A,d}\sum_\nu n_\nu^2
	\le r\sum_{\ell,\nu}n_{\nu,\ell}^2
	\le r\sum_{\ell=1}^r B_\ell(\delta).
	\]
	The same argument applies when some coordinates lie in $\T$, using
	periodic distance and intervals of length comparable to
	$\min(\delta_i,1)$. All counts retain index multiplicities.
	
	\section{The double large sieve}\label{sec:sieve}
	\begin{lemma}[Double large sieve]\label{lem:elementary-sieve}
		Let $d,n\in\N$, let $D_i,E_i>0$ for $1\le i\le d$, and let
		$\mathcal I$ be finite. Suppose
		$x_j=(x_{j,1},\ldots,x_{j,d})\in\mathcal X=\prod_{i=1}^d[-D_i,D_i]$
		for $1\le j\le n$, and
		$y_\nu=(y_{\nu,1},\ldots,y_{\nu,d})\in\prod_{i=1}^d[-E_i,E_i]$
		for $\nu\in\mathcal I$.
		For $x\in\R^d$ and complex coefficients $b=(b_\nu)_{\nu\in\mathcal I}$
		with $|b_\nu|\le1$, put
		\[
		\begin{gathered}
			F_b(x)=\sum_{\nu\in\mathcal I}b_\nu e(x\cdot y_\nu),\qquad
			P=\prod_{i=1}^d(1+D_iE_i),\\
			B=\#\{(j,j')\in\{1,\ldots,n\}^2:
			|x_{j,i}-x_{j',i}|\le E_i^{-1}\ (1\le i\le d)\}.
		\end{gathered}
		\]
		For $q>2$, set $\mathcal A_q=\sup_{|b_\nu|\le1}
		\|F_b\|_{L^q_\#(\mathcal X)}$. For complex coefficients $a_1,\ldots,a_n$
		with $|a_j|\le1$,
		\begin{equation}\label{eq:elementary-sieve}
			\left|\sum_{j=1}^na_jF_b(x_j)\right|
			\lesssim_{d,q}P^{1/q}n^{1-2/q}B^{1/q}\mathcal A_q.
		\end{equation}
		If $c_4,c_6\ge0$ satisfy, for every $\lambda>0$,
		\[
		\sup_{|b_\nu|\le1}
		\frac{|\{x\in\mathcal X:|F_b(x)|>\lambda\}|}{|\mathcal X|}
		\le c_4\lambda^{-4}+c_6\lambda^{-6},
		\]
		then
		\begin{equation}\label{eq:distributional-sieve}
			\left|\sum_{j=1}^na_jF_b(x_j)\right|
			\lesssim_d\sum_{q\in\{4,6\}}P^{1/q}n^{1-2/q}B^{1/q}c_q^{1/q}.
		\end{equation}
		The count $B$ includes all ordered pairs of indices, with their multiplicities.
	\end{lemma}
	\begin{proof}
		The moment estimate \eqref{eq:elementary-sieve} is
		\cite[\Extref{BW18}{Section~5}, (5.5)]{BW18}, with bounded coefficients as in
		\cite[\Extref{BW18}{Section~9}, (9.1)--(9.4)]{BW18}.
		For \eqref{eq:distributional-sieve}, partition $\mathcal X$ into
		axis-parallel boxes $\mathcal Q$ of side lengths
		\[
		\delta_i=\frac{2D_i}{\max\{1,\lceil2D_iE_i\rceil\}},\qquad
		\delta_iE_i\le1,\qquad |\mathcal Q|^{-1}\lesssim_d P/|\mathcal X|.
		\]
		Assign each $x_j$ to one box containing it, and let $n_{\mathcal Q}$ be the
		number of assigned indices. Then
		\[
		\sum_{\mathcal Q} n_{\mathcal Q}=n,\qquad \sum_{\mathcal Q}n_{\mathcal Q}^2\le B.
		\]
		Fix the coefficient sequence $b$.
		For $\mu\in\{0,1\}^d$, define
		\[
		F_{b,\mu}(x)=\sum_{\nu\in\mathcal I}b_\nu
		\prod_{i=1}^d(y_{\nu,i}/E_i)^{\mu_i}e(x\cdot y_\nu).
		\]
		These coefficients have modulus at most one. With
		$\partial^\mu=\partial_{x_1}^{\mu_1}\cdots\partial_{x_d}^{\mu_d}$,
		\[
		\partial^\mu F_b=\left(\prod_{i=1}^d(2\pi\mathrm i E_i)^{\mu_i}\right)F_{b,\mu}.
		\]
		Applying the fundamental theorem of calculus in each coordinate gives,
		for $x_j\in \mathcal Q$,
		\[
		|F_b(x_j)|\le\frac1{|\mathcal Q|}\sum_{\mu\in\{0,1\}^d}
		\prod_{i=1}^d\delta_i^{\mu_i}\int_{\mathcal Q}|\partial^\mu F_b|
		\lesssim_d\frac1{|\mathcal Q|}\sum_{\mu\in\{0,1\}^d}\int_{\mathcal Q}|F_{b,\mu}|.
		\]
		For $\lambda>0$, integration of the distribution function gives
		\[
		\begin{aligned}
			\int_{\mathcal X\cap\{|F_{b,\mu}|>\lambda\}}|F_{b,\mu}|^2
			&=\lambda^2|\mathcal X\cap\{|F_{b,\mu}|>\lambda\}|
			+2\int_\lambda^\infty t|\mathcal X\cap\{|F_{b,\mu}|>t\}|\,dt\\
			&\lesssim|\mathcal X|\bigl(c_4\lambda^{-2}+c_6\lambda^{-4}\bigr).
		\end{aligned}
		\]
		Split each integral over $\mathcal Q$ at $|F_{b,\mu}|=\lambda$ and use
		Cauchy--Schwarz. For every $\lambda>0$,
		\begin{equation}\label{eq:sieve-amplitude-split}
			\begin{split}
				\left|\sum_{j=1}^na_jF_b(x_j)\right|
				&\lesssim_d n\lambda+
				\left(\sum_{\mathcal Q}\frac{n_{\mathcal Q}^2}{|\mathcal Q|}\right)^{1/2}
				\sum_{\mu\in\{0,1\}^d}
				\left(\int_{\mathcal X\cap\{|F_{b,\mu}|>\lambda\}}
				|F_{b,\mu}|^2\right)^{1/2}\\
				&\lesssim_d n\lambda+(PB)^{1/2}
				\big(c_4\lambda^{-2}+c_6\lambda^{-4}\big)^{1/2}.
			\end{split}
		\end{equation}
		If $c_4=c_6=0$, let $\lambda\downarrow0$ in
		\eqref{eq:sieve-amplitude-split}. Otherwise take
		$\lambda=(PBc_4/n^2)^{1/4}+(PBc_6/n^2)^{1/6}>0$.
		For $q\in\{4,6\}$, $PBc_q\lambda^{2-q}\le n^2\lambda^2$,
		which proves \eqref{eq:distributional-sieve}.
	\end{proof}
	
	\medskip
	\noindent\textbf{Funding}
	C.Z. is partially supported by the National Key R\&D Program of China
	No.~2024YFA1015300 and NSFC Grant No.~12371097. J.Z. is partially supported by NSFC Grant No.~12531005.
	
	\medskip
	\noindent\textbf{Conflicts of Interest}\space The authors declare that there are no conflicts of interest.
	
	\medskip
	
	\noindent{\bf AI disclosure.}
	During the preparation of this manuscript, the authors used ChatGPT and Codex to discuss possible proof strategies, assist in drafting and checking parts of the mathematical arguments, organize references, and improve the clarity and readability of the exposition. All mathematical statements and proofs in the final manuscript were written, reviewed, and verified by the authors, who take full responsibility for the content of this manuscript.

\end{document}